\documentclass[11pt]{amsart}
\usepackage{amsopn}
\usepackage{amssymb, amscd}
\usepackage{multirow}
\usepackage{graphicx, graphics, epsfig}
\usepackage{faktor}
\usepackage{enumerate}
\usepackage{tikz}
\usepackage{pgfplots}
\usepackage{hyperref}
\usepackage{color}
\usepackage{appendix}

\newcommand{\nc}{\newcommand}

\nc{\fg}{\mathfrak{f} } \nc{\vg}{\mathfrak{v} } \nc{\wg}{\mathfrak{w} }
\nc{\zg}{\mathfrak{z} } \nc{\ngo}{\mathfrak{n} } \nc{\kg}{\mathfrak{k} }
\nc{\mg}{\mathfrak{m} } \nc{\bg}{\mathfrak{b} } \nc{\ggo}{\mathfrak{g} } \nc{\eg}{\mathfrak{e} }
\nc{\ggob}{\overline{\mathfrak{g}} } \nc{\sog}{\mathfrak{so} }
\nc{\sug}{\mathfrak{su} } \nc{\spg}{\mathfrak{sp} } \nc{\slg}{\mathfrak{sl} }
\nc{\glg}{\mathfrak{gl} } \nc{\cg}{\mathfrak{c} } \nc{\rg}{\mathfrak{r} }
\nc{\hg}{\mathfrak{h} } \nc{\tg}{\mathfrak{t} } \nc{\ug}{\mathfrak{u} }
\nc{\dg}{\mathfrak{d} } \nc{\ag}{\mathfrak{a} } \nc{\pg}{\mathfrak{p} }
\nc{\sg}{\mathfrak{s} } \nc{\affg}{\mathfrak{aff} } \nc{\qg}{\mathfrak{q} } \nc{\lgo}{\mathfrak{l} }

\nc{\pca}{\mathcal{P}} \nc{\nca}{\mathcal{N}} \nc{\lca}{\mathcal{L}}
\nc{\oca}{\mathcal{O}} \nc{\mca}{\mathcal{M}} \nc{\tca}{\mathcal{T}}
\nc{\aca}{\mathcal{A}} \nc{\cca}{\mathcal{C}} \nc{\gca}{\mathcal{G}}
\nc{\sca}{\mathcal{S}} \nc{\hca}{\mathcal{H}} \nc{\bca}{\mathcal{B}}
\nc{\dca}{\mathcal{D}} \nc{\eca}{\mathcal{E}} \nc{\wca}{\mathcal{W}}

\nc{\vp}{\varphi} \nc{\ddt}{\tfrac{d}{dt}} \nc{\dsdt}{\tfrac{d^2}{dt^2}} \nc{\dds}{\frac{d}{ds}}
\nc{\dpar}{\frac{\partial}{\partial t}} \nc{\im}{\mathrm{i}}

\nc{\SO}{\mathrm{SO}} \nc{\Spe}{\mathrm{Sp}} \nc{\Sl}{\mathrm{SL}}
\nc{\SU}{\mathrm{SU}} \nc{\Or}{\mathrm{O}} \nc{\U}{\mathrm{U}} \nc{\Gl}{\mathrm{GL}}
\nc{\Se}{\mathrm{S}} \nc{\Cl}{\mathrm{Cl}} \nc{\Spin}{\mathrm{Spin}}
\nc{\Pin}{\mathrm{Pin}} \nc{\G}{\mathrm{GL}_n(\RR)} \nc{\g}{\mathfrak{gl}_n(\RR)}

\nc{\RR}{{\Bbb R}} \nc{\HH}{{\Bbb H}} \nc{\CC}{{\Bbb C}} \nc{\ZZ}{{\Bbb Z}}
\nc{\FF}{{\Bbb F}} \nc{\NN}{{\Bbb N}} \nc{\QQ}{{\Bbb Q}} \nc{\PP}{{\Bbb P}} \nc{\OO}{{\Bbb O}}

\nc{\vs}{\vspace{.2cm}} \nc{\vsp}{\vspace{1cm}} \nc{\ip}{\langle\cdot,\cdot\rangle}
\nc{\ipp}{(\cdot,\cdot)} \nc{\la}{\langle} \nc{\ra}{\rangle} \nc{\unm}{\tfrac{1}{2}}
\nc{\unc}{\tfrac{1}{4}} \nc{\und}{\frac{1}{16}} \nc{\no}{\vs\noindent}
\nc{\lam}{\Lambda^2(\RR^n)^*\otimes\RR^n} \nc{\tangz}{{\rm T}^{\rm Zar}}
\nc{\nor}{{\sf n}}  \nc{\mum}{/\!\!/} \nc{\kir}{/\!\!/\!\!/}
\nc{\Ri}{\tfrac{4\Ric_{\mu}}{||\mu||^2}} \nc{\ds}{\displaystyle}
\nc{\ben}{\begin{enumerate}} \nc{\een}{\end{enumerate}} \nc{\f}{\frac}
\nc{\lb}{[\cdot,\cdot]} \nc{\isn}{\tfrac{1}{||v||^2}}
\nc{\gkp}{(\ggo=\kg\oplus\pg,\ip)} \nc{\ukh}{(\ug=\kg\oplus\hg,\ip)}
\nc{\tgkp}{(\tilde{\ggo}=\kg\oplus\pg,\ip)}
\nc{\wt}{\widetilde}
\nc{\iop}{\mathtt{i}} \nc{\jop}{\mathtt{j}} 
\nc{\Hk}{H_{\kil}} \nc{\gk}{g_{\kil}}

\nc{\Hess}{\operatorname{Hess}} \nc{\ad}{\operatorname{ad}}
\nc{\Ad}{\operatorname{Ad}} \nc{\rank}{\operatorname{rk}}
\nc{\Irr}{\operatorname{Irr}} \nc{\End}{\operatorname{End}}
\nc{\Aut}{\operatorname{Aut}} \nc{\Inn}{\operatorname{Inn}}
\nc{\Der}{\operatorname{Der}} \nc{\Ker}{\operatorname{Ker}}
\nc{\Iso}{\operatorname{Iso}} \nc{\Diff}{\operatorname{Diff}}
\nc{\Lie}{\operatorname{L}} \nc{\tr}{\operatorname{tr}} \nc{\dif}{\operatorname{d}}
\nc{\sen}{\operatorname{sen}} \nc{\modu}{\operatorname{mod}}
\nc{\CRic}{\operatorname{PP}} \nc{\Cric}{\operatorname{P}} \nc{\Ricci}{\operatorname{Ric}}
\nc{\sym}{\operatorname{sym}} \nc{\herm}{\operatorname{herm}} \nc{\symac}{\operatorname{sym^{ac}}}
\nc{\symc}{\operatorname{sym^{c}}} \nc{\scalar}{\operatorname{Sc}}
\nc{\grad}{\operatorname{grad}} \nc{\ricci}{\operatorname{Rc}} \nc{\kil}{\operatorname{B}} \nc{\cas}{\operatorname{C}} \nc{\lic}{\operatorname{L}}
\nc{\Nor}{\operatorname{Norm}}  \nc{\ricc}{\operatorname{Rc^{c}}}
\nc{\Ricc}{\operatorname{Ric^{c}}} \nc{\ricac}{\operatorname{Rc^{ac}}}
\nc{\Ricac}{\operatorname{Ric^{ac}}} \nc{\Riem}{\operatorname{Rm}} \nc{\Sec}{\operatorname{Sec}}
\nc{\riccig}{\operatorname{ric^{\gamma}}} \nc{\mm}{\operatorname{m}} \nc{\Mm}{\operatorname{M}}
\nc{\Le}{\operatorname{L}} \nc{\tang}{\operatorname{T}}
\nc{\level}{\operatorname{level}} \nc{\rad}{\operatorname{r}}
\nc{\abel}{\operatorname{ab}} \nc{\CH}{\operatorname{CH}} \nc{\Cone}{{\mathcal C}} \nc{\CCone}{\operatorname{CC}} \nc{\CP}{{\mathcal P}}
\nc{\mcc}{\operatorname{mcc}} \nc{\Adj}{\operatorname{Adj}}
\nc{\Order}{\operatorname{O}}  \nc{\inj}{\operatorname{inj}} \nc{\proy}{\operatorname{pr}}
\nc{\vol}{\operatorname{vol}} \nc{\Diag}{\operatorname{Dg}} \nc{\Diagg}{\operatorname{Diag}}
\nc{\Spec}{\operatorname{Spec}} \nc{\Ima}{\operatorname{Im}} \nc{\Rea}{\operatorname{Re}}
\nc{\spann}{\operatorname{span}} \nc{\Aff}{\operatorname{Aff}} \nc{\E}{\operatorname{E}} \nc{\id}{\operatorname{id}} \nc{\dete}{\operatorname{det}} \nc{\Crit}{\operatorname{Crit}} \nc{\val}{\operatorname{val}}

\theoremstyle{plain}
\newtheorem{theorem}{Theorem}[section]
\newtheorem{proposition}[theorem]{Proposition}

\newtheorem{lemma}[theorem]{Lemma}

\theoremstyle{definition}
\newtheorem{definition}[theorem]{Definition}

\newtheorem{assumption}[theorem]{Assumption}

\theoremstyle{remark}
\newtheorem{remark}[theorem]{Remark}
\newtheorem{example}[theorem]{Example}

\title[Bismut Ricci flat metrics on homogeneous spaces]{Bismut Ricci flat generalized metrics on compact homogeneous spaces (including a Corrigendum)}

\author{Jorge Lauret}  \author{Cynthia Will}

\address{FaMAF, Universidad Nacional de C\'ordoba and CIEM, CONICET (Argentina)}
\email{jorgelauret@unc.edu.ar} \email{cynthia.will@unc.edu.ar}

\thanks{This research was partially supported by grants from Univ. Nac. de C\'ordoba and Foncyt (Argentina)}

\date{\today}

\begin{document}

\maketitle

\begin{abstract}
A generalized metric on a manifold $M$, i.e., a pair $(g,H)$, where $g$ is a Riemannian metric and $H$ a closed $3$-form, is a fixed point of the generalized Ricci flow if and only if $(g,H)$ is Bismut Ricci flat: $H$ is $g$-harmonic and $\ricci(g)=\unc H_g^2$.  On any homogeneous space $M=G/K$, where $G=G_1\times G_2$ is a compact semisimple Lie group with two simple factors, under some mild assumptions, we exhibit a Bismut Ricci flat $G$-invariant generalized metric, which is proved to be unique among a $4$-parameter space of metrics in many cases, including when $K$ is neither abelian nor semisimple.  On the other hand, if $K$ is simple and the standard metric is Einstein on both $G_1/\pi_1(K)$ and $G_2/\pi_2(K)$, we give a one-parameter family of Bismut Ricci flat $G$-invariant generalized metrics on $G/K$ and show that it is most likely pairwise non-homothetic by computing the ratio of Ricci eigenvalues.  This is proved to be the case for every space of the form $M=G\times G/\Delta K$ and for $M^{35}=\SO(8)\times\SO(7)/G_2$.   

A Corrigendum has been added in Appendix \ref{app}.  
\end{abstract}

\tableofcontents

\section{Introduction}\label{intro}

In generalized Riemannian geometry, a {\it generalized metric} on an exact Courant algebroid over a manifold $M$ is encoded in a pair $(g,H)$, where $g$ is a Riemannian metric and $H$ a closed $3$-form on $M$.  A very natural direction to make these structures evolve is provided by the Ricci tensor of the corresponding {\it Bismut connection}, i.e., the unique metric connection on $(M,g)$ with torsion $3$-covariant tensor $g(T_\cdot\cdot,\cdot)$ equal to $H$, giving rise to the following evolution equation for a family of generalized metrics $(g(t),H(t))$ called the {\it generalized Ricci flow}:
\begin{equation}\label{GRF}
\left\{\begin{array}{l}
\dpar g(t) = -2\ricci(g(t)) +\unm (H(t))_{g(t)}^2, \\ \\
\dpar H(t) = -dd_{g(t)}^*H(t),
\end{array}\right.
\end{equation}
where $H_g^2:=g(\iota_\cdot H,\iota_\cdot H)$.  We refer to the recent book \cite{GrcStr} and the references given in \S\ref{preli1} for further mathematical and physical motivation, including applications to the pluriclosed flow of SKT structures and the generalized K\"ahler Ricci flow.       

As for any curvature flow, in order to understand its behavior, it is crucial to study the existence of fixed points, which in the case of the generalized Ricci flow are precisely {\it Bismut Ricci flat} (BRF for short) generalized metrics $(g,H)$ (or {\it generalized Einstein metrics}), i.e.,  
\begin{equation}\label{BRF-def-intro}
\ricci(g)=\unc H_g^2 \quad\mbox{and}\quad \mbox{$H$ is $g$-harmonic}.  
\end{equation} 
Note that $(g,0)$ with $g$ Ricci flat provide trivial examples, so we require $H\ne 0$ from now on, which implies the necessary topological condition $b_3(M)>0$. 

Examples of BRF are of course provided by generalized metrics such that the Bismut connection is {\it flat}, although this turns out to be extremely strong.  It is well known that the only compact simply connected Riemannian manifolds carrying a flat Bismut connection with closed torsion $3$-form are compact semisimple Lie groups endowed with a bi-invariant generalized metric of the form $(g_b,\pm H_b)$, where $g_b$ is a bi-invariant metric and $H_b:=g_b([\cdot,\cdot],\cdot)$, the attached Cartan $3$-form (see \cite{CrtKrs, AgrFrd, GrcStr}).  On the contrary, the study of non-flat BRF generalized metrics is only in its infancy, even in the homogeneous case.  It is unknown whether a compact simple Lie group admits a left-invariant BRF generalized metric different from the bi-invariant one or not.  The only other known homogeneous BRF generalized metrics were recently given by Podest\`a and Raffero in \cite{PdsRff1} on $S^2\times S^3$, one for each homogeneous space presentation $\SU(2)\times\SU(2)/S^1_{p,q}$, and in \cite{PdsRff2}, one on any homogeneous space of the form $G\times G/\Delta K$, where $(G,K)$ is a symmetric pair such that $\rank G=\rank K$ and $\Delta K$ denotes the usual diagonal embedding.  Note that BRF condition \eqref{BRF-def-intro} implies that $\ricci(g)\geq 0$, which essentially gives that $M$ is necessarily compact in the homogeneous setting.   

The topological obstruction $b_3(M)>0$ is stronger than one may think for a compact homogeneous manifold.  According to \cite{H3}, if $M=G/K$ is a homogeneous space, where $G$ is compact and semisimple with $s$ simple factors, then $b_3(M)\leq s-1$, unless $K$ is trivial (in particular, $G$ simple and $K$ non-trivial is excluded), and in the case when $K$ projects non-trivially on all the factors of $G$, equality $b_3(M)=s-1$ holds if and only if $G/K$ is {\it aligned}: the Killing forms satisfy that 
\begin{equation*} 
\kil_{\ggo_i}(Z_i,W_i) = \tfrac{1}{c_i}\kil_\ggo(Z,W), \qquad\forall Z,W\in\kg, \quad i=1,\dots,s,
\end{equation*}   
for some fixed $c_1,\dots,c_s>0$, where $\ggo=\ggo_1\oplus\dots\oplus\ggo_s$ is the decomposition of the Lie algebra of $G$ in simple factors.  Note that $G/K$ is automatically aligned if $\kg$ is simple or one-dimensional. 

In this paper we study, via a unified approach, the existence of BRF generalized metrics on any homogeneous space of a semisimple Lie group with two simple factors.  This case already demands very heavy computations, so we have postponed the case of more simple factors for future work.  The class of homogeneous spaces considered is huge and unclassifiable, due to the several different possible embeddings for the isotropy subgroups.   

We therefore consider $M=G/K$, where $G=G_1\times G_2$ and $G_1$, $G_2$ are compact simple Lie groups.  It is natural to assume that $K$ projects non-trivially on both $G_1$ and $G_2$, otherwise $M$ is the product of a homogeneous space and a Lie group.  We also have to assume that $b_3(M)>0$, which leads us to an aligned $M=G/K$ with $b_3(M)=1$.  In particular, $\pi_i(\kg)\simeq\kg$ for $i=1,2$.  The following technical property is also assumed: as $\Ad(K)$-representations, none of the irreducible components of the isotropy representations of $G_1/\pi_1(K)$ and $G_2/\pi_2(K)$ is equivalent to any of the simple or $1$-dimensional factors of $\kg$.  This extra assumption might not be necessary, we use it to apply \cite[Section 6.1]{H3} in order to obtain that certain canonical closed $3$-form is harmonic with respect to any metric in our search $4$-parameter space of $G$-invariant metrics.  The following is the first main result of the paper.   

\begin{theorem}\label{main} (See Theorem \ref{BRFggs2}).
Let $M=G/K$ be a homogeneous space as above.  Then $M$ admits a $G$-invariant Bismut Ricci flat generalized metric.  
\end{theorem}

The lowest dimensional new example is $M^{10}=\SU(2)\times\SU(3)/S^1$.  The search space of generalized metrics on each $M=G/K$ is given as follows.  For any background bi-invariant metric on $G$, say
\begin{equation*}
g_b=z_1(-\kil_{\ggo_1})+\tfrac{1}{c_1-1}(-\kil_{\ggo_2}), \qquad z_1>0,   
\end{equation*}
we consider the $g_b$-orthogonal reductive decomposition $\ggo=\kg\oplus\pg$, the normal metric on $M=G/K$ determined by $g_b|_{\pg\times\pg}$ and the $g_b$-orthogonal $\Ad(K)$-invariant decomposition 
\begin{equation}\label{dec1}
\pg=\pg_1\oplus\pg_2\oplus\pg_3,
\end{equation}
where $\pg_i$ is equivalent to the isotropy representation of the homogeneous space $M_i=G_i/\pi_i(K)$ for $i=1,2$ and $\pg_3$ is equivalent to the adjoint representation $\kg$.  According to \cite[Section 6.1]{H3}, the $G$-invariant $3$-form given by
\begin{equation*}
H_0(X,Y,Z) := Q([X,Y],Z) + Q([X,Y]_\kg,Z) - Q([X,Z]_\kg,Y) + Q([Y,Z]_\kg,X),
\end{equation*}
for all $X,Y,Z\in\pg$, where $Q=\kil_{\ggo_1}-\tfrac{1}{c_1-1}\kil_{\ggo_2}$, is $g$-harmonic for any $G$-invariant metric $g=(x_1,x_2,x_3)_{g_b}$ defined by 
\begin{equation*}
g=x_1g_b|_{\pg_1\times\pg_1}+x_2g_b|_{\pg_2\times\pg_2}+x_3g_b|_{\pg_3\times\pg_3}, \qquad x_1,x_2,x_3>0.
\end{equation*} 
Note that $H_0$ also depends on $z_1$ since the projection $[X,Y]_\kg$ of $[X,Y]$ on $\kg$ is taken with respect to the $g_b$-orthogonal decomposition $\ggo=\kg\oplus\pg$.  This explains why we have let the background metric $g_b$ vary in our search space.  In the case when $G_1/\pi_1(K)$ and $G_2/\pi_2(K)$ are both isotropy irreducible spaces, their isotropy representations are inequivalent, $[\kg,\kg]$ is simple and $\dim{\kg_0}=0,1$, these represent all $G$-invariant metrics on $M=G/K$ for each fixed $z_1>0$.  The Ricci curvature of these metrics is computed in \S\ref{rical-sec} and the symmetric form $(H_0)_g^2$ in \S\ref{HQ2-sec}.  

Explicitly, the BRF generalized metric in Theorem \ref{main} is given by $(g_0,H_0)$, where 
$$
g_0:=\left(\tfrac{1}{z_1},1,\tfrac{z_1+1}{z_1}\right)_{g_b}, \qquad z_1=c_1-1.
$$ 
Concerning uniqueness, always among structures of the form $(g=(x_1,x_2,x_3)_{g_b},H_0)$, we prove that $(g_0,H_0)$ is the only one up to scaling which is Bismut Ricci flat if $\kg$ has a nontrivial center (any $z_1>0$ actually  works when $\kg$ is abelian), as well as if $\kg$ is not abelian and the standard metric is not Einstein on $M_i=G_i/\pi_i(K)$, $i=1,2$  (see \S\ref{BRF-sec} and Theorem \ref{BRFggs2} for further details).    

On the other hand, the existence of curves is the second main result.  

\begin{theorem}\label{main2} (See Theorems \ref{BRFggs2-kss} and \ref{BRFcurve}).
Assume that $\kg$ is semisimple, $\kil_{\pi_i(\kg)}=a_i\kil_{\ggo_i}$ for some constant $a_i$, $i=1,2$ (e.g., $\kg$ simple), and the standard (or Killing) metric is Einstein on both $G_1/\pi_1(K)$ and $G_2/\pi_2(K)$. 
\begin{enumerate}[{\rm (i)}]
\item There is a one-parameter family $(g(z_1),H(z_1))$, $z_1>0$ of Bismut Ricci flat $G$-invariant generalized metrics on the homogeneous space $M=G/K$.  

\item If $G_1=G_2$ and $K$ is a simple subgroup diagonally embedded, then the curve in part (i) is pairwise non-homothetic around $(g_0,H_0)$.  
\end{enumerate}
\end{theorem}

The possibilities for the spaces $G_i/\pi_i(K)$ are described in Remarks \ref{rem2} and \ref{posib}, respectively.  
The lowest dimensional example for part (ii) is $M^{13}=\SU(3)\times\SU(3)/\SO(3)$.

The coordinates $x_1$ and $x_2$ of the metric $g(z_1)$ can be written in terms of $x_3$ in a relatively simple way, but $x_3$ is the solution of a quite involved equation in terms of $z_1$ and some algebraic constants attached to $G/K$ (see \eqref{BRF3ra-3} for further details).  This makes the problem of whether or not two metrics in the family are isometric up to scaling quite difficult to approach in a unified way, even with a formula for the ratio of Ricci eigenvalues at hand (see \S\ref{curve-sec}).  

Beyond the case considered by Theorem \ref{main2}, (ii), we also prove that the one-parameter family $(g(z_1),H(z_1))$ is indeed pairwise non-homothetic around $(g_0,H_0)$ on the homogeneous space $M^{35}=\SO(8)\times\SO(7)/G_2$ (see Example \ref{curve}). 

Finally, in \S\ref{simple-sec}, we study the case when $M=G$ is a compact Lie group, obtaining the following negative existence result:  other than $(g_b,\pm H_b)$, there are no BRF generalized metrics on $G$ of the form $(g,\pm H_b)$, where $g$ runs among all left-invariant metrics on $G$, or of the form $(g_b,H)$, where $H$ is any left-invariant closed $3$-form on $G$.

\vs \noindent {\it Acknowledgements.}  The authors are grateful with the anonymous referee for many helpful comments.

\section{Preliminaries}\label{preli}

\subsection{Generalized metrics and Bismut connection}\label{preli1}
We refer to the recent book \cite{GrcStr} and the articles \cite{CrtKrs, Grc, Lee, PdsRff1, PdsRff2, RubTpl, Str1, Str2} for further information on the subject of this subsection.  

A metric connection on a given compact Riemannian manifold $(M,g)$ is called {\it Bismut} if its torsion $T$ satisfies that the $3$-covariant tensor
$$
H(X,Y,Z) := g(T_XY,Z), \qquad\forall X,Y,Z\in\chi(M),
$$
is a $3$-form on $M$.  In that case, the Bismut connection $\nabla^B$ is given by 
$$
g(\nabla^B_XY,Z)=g(\nabla^g_XY,Z)+\unm H(X,Y,Z), \qquad\forall X,Y,Z\in\chi(M), 
$$
where $\nabla^g$ is the Levi Civita connection of $(M,g)$, and if $H$ is closed then the pair $(g,H)$ is called a {\it generalized metric}.  The Ricci tensor of $\nabla^B$ is given by 
$$
\ricci^B(g,H) = \ricci(g) - \unc H_g^2 - \unm d_g^*H,
$$
where $\ricci(g)$ is the Ricci tensor of the metric $g$, $H_g^2$ is the non-negative definite symmetric $2$-tensor defined by
$$
H_g^2(X,Y) := g(\iota_XH,\iota_YH), \qquad\forall X,Y\in\chi(M), 
$$  
and $d_g^*$ is the adjoint of $d$ with respect to $g$.  Thus $\ricci(g) - \unc H_g^2$ and $- \unm d_g^*H$ are respectively the symmetric and skew-symmetric parts of $\ricci^B(g,H)$ and so the generalized metric $(g,H)$ is {\it Bismut Ricci flat} (BRF for short), i.e., $\ricci^B(g,H)=0$, if and only if 
\begin{equation}\label{BRF-def}
\ricci(g)=\unc H_g^2 \quad\mbox{and}\quad \mbox{$H$ is $g$-harmonic}.  
\end{equation} 
Note that since $H$ is closed, $H$ is $g$-harmonic (i.e., $\Delta_gH=0$, where $\Delta_g:=dd_g^*+d_g^*d$ is the Hodge Laplacian operator) if and only if $H$ is coclosed (i.e. $d_g^*H=0$).  It is easy to see that $(g,H)$ is BRF if and only if $(cg,\pm cH)$ is BRF for any $c>0$.  

BRF metrics are also called {\it generalized Einstein metrics} and are precisely the fixed points of the {\it generalized Ricci flow}, given by the evolution equation \eqref{GRF} (see \cite{GrcStr, Lee} and references therein).

\subsection{Aligned homogeneous spaces}\label{preli3}
In this section we review the class of homogeneous spaces with the richest third cohomology (other than Lie groups), i.e., $b_3(G/K)=s-1$ if $G$ has $s$ simple factors, called {\it aligned} homogeneous spaces.  See \cite{H3} for a more detailed treatment.  

Let $M=G/K$ be a homogeneous space, where $G$ is a compact, connected and semisimple Lie group and $K$ is a connected closed subgroup.  We fix decompositions for the corresponding Lie algebras, 
\begin{equation}\label{decss}
\ggo=\ggo_1\oplus\dots\oplus\ggo_s, \qquad \kg=\kg_0\oplus\kg_1\oplus\dots\oplus\kg_t, 
\end{equation}
where the $\ggo_i$'s and $\kg_j$'s are simple ideals of $\ggo$ and $\kg$, respectively, and $\kg_0$ is the center of $\kg$.  If $\pi_i:\ggo\rightarrow\ggo_i$ are the usual projections, then we set $Z_i:=\pi_i(Z)$ for any $Z\in\ggo$.  The Killing form of a Lie algebra $\hg$ will always be denoted by $\kil_\hg$. 

\begin{definition}\label{alig-def-2}
A homogeneous space $G/K$ as above is said to be {\it aligned} when there exist $c_1,\dots,c_s>0$ such that
\begin{equation}\label{al2} 
\kil_{\ggo_i}(Z_i,W_i) = \tfrac{1}{c_i}\kil_\ggo(Z,W), \qquad\forall Z,W\in\kg, \quad i=1,\dots,s.
\end{equation}   
In other words, the ideals $\kg_j$'s are uniformly embedded on each $\ggo_i$ in some sense: the multiple $\frac{1}{c_i}$ is independent of $j$.
\end{definition}

Note that $G/K$ is automatically aligned if $\kg$ is simple or one-dimensional.  The following properties of an aligned homogeneous space $G/K$ easily follow: 
\begin{enumerate}[{\rm (i)}]
\item $\frac{1}{c_1}+\dots+\frac{1}{c_s}=1$.  

\item $\pi_i(\kg)\simeq\kg$ for all $i=1,\dots,s$.  

\item The Killing constants, defined by 
\begin{equation}\label{cij}
\kil_{\pi_i(\kg_j)} = c_{ij}\kil_{\ggo_i}|_{\pi_i(\kg_j)\times\pi_i(\kg_j)},
\end{equation}
satisfy the following alignment property:
$$
(c_{1j},\dots,c_{sj}) = \lambda_j(c_1,\dots,c_s) \quad\mbox{for some}\quad \lambda_j>0, \qquad\forall j=1,\dots,t.  
$$
Indeed, for any $Z,W\in\kg_j$ and $i=1,\dots,s$,
$$
\kil_\kg(Z,W)=\kil_{\kg_j}(Z,W)=\kil_{\pi_i(\kg_j)}(Z_i,W_i) = c_{ij}\kil_{\ggo_i}(Z_i,W_i) =\tfrac{c_{ij}}{c_i}\kil_\ggo(Z,W).
$$ 
\item There exists an inner product $\ip$ on $\kg_0$ (given by $-\kil_{\ggo}|_{\kg_0\times\kg_0}$) such that 
$$
\kil_{\ggo_i}(Z_i,W_i)=-\tfrac{1}{c_i}\la Z,W\ra, \qquad\forall Z,W\in\kg_0.    
$$ 
\item The Killing form of $\kg_j$ is given by
\begin{equation}\label{al1}
\kil_{\kg_j}=\lambda_j\kil_{\ggo}|_{\kg_j\times\kg_j}, \qquad\forall j=1,\dots,t.  
\end{equation}
\end{enumerate}
The original definition given in \cite[Definition 4.7]{H3} consists of parts (i), (iii) and (iv) above.  Under the assumption that $\pi_i(\kg)\ne 0$ for all $i=1,\dots,s$, $G/K$ is aligned if and only if 
there exists an inner product $\ip$ on $\kg$ such that $Q|_{\kg\times\kg}$ coincides with $\ip$ up to scaling for any bi-invariant symmetric bilinear form $Q$ on $\ggo$ (see \cite[Proposition 4.10]{H3}).   

We assume from now on that $M=G/K$ is an aligned homogeneous space with $s=2$, i.e., 
\begin{equation}\label{decs}
\ggo=\ggo_1\oplus\ggo_2, \qquad \kg=\kg_0\oplus\kg_1\oplus\dots\oplus\kg_t. 
\end{equation}
For any given bi-invariant metric on $G$, say
\begin{equation}\label{gbdef}
g_b=z_1(-\kil_{\ggo_1})+z_2(-\kil_{\ggo_2}), \qquad z_1,z_2>0,   
\end{equation}
we consider the $g_b$-orthogonal reductive decomposition $\ggo=\kg\oplus\pg$, the normal metric on $M=G/K$ determined by $g_b|_{\pg\times\pg}$ and the $g_b$-orthogonal $\Ad(K)$-invariant decomposition 
\begin{equation}\label{dec1}
\pg=\pg_1\oplus\pg_2\oplus\pg_3,
\end{equation}
provided by \cite[Proposition 5.1]{H3}.  As $\Ad(K)$-representations, $\pg_i$ is equivalent to the isotropy representation of the homogeneous space $M_i=G_i/\pi_i(K)$ for $i=1,2$, i.e., $\ggo_i=\pi_i(\kg)\oplus\pg_i$, and 
$$
\pg_3 := \left\{ \left(Z_1,-\tfrac{c_2z_1}{c_1z_2}Z_2\right):Z\in\kg\right\},
$$ 
so $\pg_3$ is equivalent to the adjoint representation $\kg$.  Under the following assumption, any $G$-invariant $2$-form $\omega$ necessarily satisfies that $\omega(\pg_i,\pg_3)=0$ for $i=1,2$.  

\begin{assumption}\label{assum}
None of the irreducible components of $\pg_1,\pg_2$ is equivalent to any of the simple factors of $\kg$ as $\Ad(K)$-representations and either $\zg(\kg)=0$ or the trivial representation is not contained in any of $\pg_1,\pg_2$ (see \cite[Section 6]{H3} for more details on this assumption).
\end{assumption}

According to \cite[Proposition 4.1 and (12)]{H3}, each bi-invariant symmetric bilinear form on $\ggo$, 
\begin{equation}\label{Qdef}
Q=y_1\kil_{\ggo_1}+y_2\kil_{\ggo_2} \quad\mbox{such that}\quad 
\tfrac{y_1}{c_1}+\tfrac{y_2}{c_2}=0 \quad(\mbox{i.e.},\; Q|_{\kg\times\kg}=0),
\end{equation}
defines a $G$-invariant closed $3$-form $H_Q$ on $M$ by
\begin{equation}\label{HQ-def}
H_Q(X,Y,Z) := Q([X,Y],Z) + Q([X,Y]_\kg,Z) - Q([X,Z]_\kg,Y) + Q([Y,Z]_\kg,X),
\end{equation}
for all $X,Y,Z\in\pg$, where $\lb_\kg$ denotes the projection of the Lie bracket $\lb$ on $\kg$ with respect to $\ggo=\kg\oplus\pg$.  Moreover, every class in $H^3(G/K)$ has a unique representative of the form $H_Q$, which implies that $b_3(G/K)=1$.  

A $G$-invariant metric of the form
\begin{equation}\label{gg-harm}
g=x_1g_b|_{\pg_1\times\pg_1}+x_2g_b|_{\pg_2\times\pg_2}+x_3g_b|_{\pg_3\times\pg_3}, \qquad x_1,x_2,x_3>0,
\end{equation} 
will be denoted by $g=(x_1,x_2,x_3)_{g_b}$.  The following result is proved in 
\cite[Section 6.1]{H3} (see also \cite[Theorem 7.4, Remark 7.5]{H3}).  

\begin{proposition}\label{harms2}
If Assumption \ref{assum} holds for an aligned $M=G/K$ with $s=2$, then every closed $3$-form $H_Q$ is harmonic with respect to any metric $g=(x_1,x_2,x_3)_{g_b}$.  
\end{proposition}

In what follows, we summarize a number of technical properties which will be used throughout the paper.  We first note that 
\begin{equation}\label{pi1pi2}
\pi_1(\kg)\oplus\pi_2(\kg) = \pg_3\oplus\kg,  
\end{equation}
is a Lie subalgebra of $\ggo$, which is abelian if and only if $\kg$ is abelian.  The only nonzero brackets between the $\pg_i$'s are the following: 
\begin{align}
&[\pg_1,\pg_1]\subset\pg_1+\pg_3+\kg, \label{pipj0}\\  
&[\pg_2,\pg_2]\subset\pg_2+\pg_3+\kg, \label{pipj1}\\  
&[\pg_3,\pg_1]\subset\pg_1, \qquad [\pg_3,\pg_2]\subset\pg_2, \label{pipj2}, \\ 
&[\pg_3,\pg_3] \subset\pg_3+\kg. \label{pipj3}
\end{align}

\begin{remark}\label{HQnz}
Since the bi-invariant form $Q$ introduced in \eqref{Qdef} satisfies that
\begin{align}
&Q|_{\pg_1\times\pg_1} =-\tfrac{y_1}{z_1}g_b, \qquad 
Q|_{\pg_2\times\pg_2} =-\tfrac{y_2}{z_2}g_b,  \label{pipj2} \\ 
&Q(\pg_1,\pg_2)=Q(\pg_1,\pg_3)=Q(\pg_2,\pg_3)=Q(\pg_1,\kg)=Q(\pg_2,\kg)=0, \notag
\end{align}
the $3$-form $H_Q$ is nonzero (up to permutations) only on 
$$
\pg_1\times\pg_1\times\pg_1, \qquad \pg_2\times\pg_2\times\pg_2, \qquad\pg_1\times\pg_1\times\pg_3,  \qquad\pg_2\times\pg_2\times\pg_3, \qquad \pg_3\times\pg_3\times\pg_3, 
$$ 
and 
\begin{equation}\label{HQ1}
H_Q(X,Y,Z) = Q([X,Y],Z) = -\tfrac{y_i}{z_i}g_b([X,Y],Z) = y_i\kil_{\ggo_i}([X,Y],Z),     
\end{equation}
for all $X,Y,Z\in\pg_i$, $i=1,2$.  
\end{remark}

The subspace $\pg_3$ in turn admits an $\Ad(K)$-invariant decomposition 
\begin{equation}\label{pjdec}
\pg_3=\pg_3^0\oplus\pg_3^1\oplus\dots\oplus\pg_3^t, \qquad \pg_3^l := \left\{ \left(Z_1,-\tfrac{c_2z_1}{c_1z_2}Z_2\right):Z\in\kg_l\right\},
\end{equation}
which is also orthogonal with respect to both $g_b$ and $Q$ and for any $l=0,\dots,t$, the subspace $\pg_3^l$ is equivalent to the adjoint representation $\kg_l$ as an $\Ad(K)$-representation (see \cite[Proposition 5.1]{H3}); in particular, $\pg_3^l$ is $\Ad(K)$-irreducible for any $1\leq l$.  Thus $\pg_3^l$ and $\pg_3^m$, $l\ne m$, are necessarily orthogonal with respect to any $\Ad(K)$-invariant symmetric bilinear form, including $g_b$, $g$, $\ricci(g)$, $Q$ and $(H_Q)_g^2$.  

The following notation will be strongly used from now on without any further reference: 
$$
A_3:=-\tfrac{c_2z_1}{c_1z_2}, \qquad B_3:=\tfrac{z_1}{c_1}+A_3^2\tfrac{z_2}{c_2},  \qquad
C_3:=\tfrac{y_1}{c_1}+A_3\tfrac{y_2}{c_2}, \qquad
B_4:=\tfrac{z_1}{c_1}+\tfrac{z_2}{c_2}.  
$$ 
We set 
$$
\ip:=-\kil_{\ggo}|_{\kg\times\kg},
$$ 
and fix a $\ip$-orthonormal basis $\{ Z^\alpha\}_{\alpha=1}^{\dim{\kg}}$ of $\kg$ adapted to the $\ip$-orthogonal decomposition $\kg=\kg_0\oplus\dots\oplus\kg_t$.  It follows from \cite[(31),(32)]{H3} that 
$$
g_b|_{\kg\times\kg}=B_4\ip,
$$ 
and 
\begin{equation}\label{gbpi}
g_b(Z_i,W_i) = \tfrac{z_i}{c_i}\la Z,W\ra, \qquad\forall Z,W\in\kg.  
\end{equation}
We also consider the $g_b$-orthonormal basis $\{e_\alpha^3\}_{\alpha=1}^{\dim{\kg}}$ of $\pg_3$ defined by 
\begin{equation}\label{basis}
e_\alpha^3:=\tfrac{1}{\sqrt{B_3}}(Z^\alpha_1,A_3Z^\alpha_2), \qquad  Z^\alpha\in\kg,
\end{equation}
which is compatible with decomposition \eqref{pjdec}.  Note that 
$$
\left\{ e_\alpha^4:=\tfrac{1}{\sqrt{B_4}}Z^\alpha\right\}_{\alpha=1}^{\dim{\kg}},
$$ 
is a $g_b$-orthonormal basis of $\kg$.  The union of all these bases together with $g_b$-orthonormal bases $\{e_\alpha^i\}_{\alpha=1}^{\dim{\pg_i}}$ of $\pg_i$, $i=1,2$, form the $g_b$-orthonormal basis of $\ggo$ which will be used in the computations.  

The following property will be useful (see \cite[Lemma 7.1,(ii)]{H3}): 
\begin{equation}\label{L71}
Q(e_\alpha^4,e_\beta^3) = \delta_{\alpha\beta} \tfrac{-C_3}{\sqrt{B_3B_4}},  
\end{equation}
where $\delta_{\alpha\beta}$ is the usual Kronecker delta.

\section{Ricci curvature}\label{rical-sec}

Let $M$ be a compact connected homogeneous manifold and fix an almost-effective transitive action of a compact connected Lie group $G$ on $M$.  The $G$-action determines a presentation $M=G/K$ of $M$ as a homogeneous space, where $K\subset G$ is the isotropy subgroup at some point $o\in M$. Given any reductive decomposition $\ggo=\kg\oplus\pg$, giving rise to the usual identification $T_oM\equiv\pg$, where $\ggo$ and $\kg$ are the Lie algebras of $G$ and $K$, respectively, we identify any $G$-invariant metric $g$ on $M$ with the corresponding $\Ad(K)$-invariant inner product on $\pg$, which will also be denoted by $g$.  

The Ricci operator $\Ricci(g):\pg\rightarrow\pg$ of a $G$-invariant metric $g$ is given by
\begin{equation}\label{Ric2}
\Ricci(g) = \Mm(g) - \unm\kil(g),
\end{equation}
where $g(\kil(g)\cdot,\cdot):=\kil_{\ggo}|_{\pg\times\pg}$ and
\begin{equation}\label{mm4}
g(\Mm(g) X,X) = -\unm\sum_{i,j} g([X,X_i]_\pg,X_j)^2+ \unc\sum_{i,j} g([X_i,X_j]_\pg,X)^2, \qquad\forall X\in\pg,
\end{equation} 
where $\{ X_i\}$ is any $g$-orthonormal basis of $\pg$ (see \cite[7.38]{Bss} or \cite[Section 2.2]{stab-dos}).  Here $\lb_\pg$ denotes the projection of the Lie bracket of $\ggo$ on $\pg$ relative to $\ggo=\kg\oplus\pg$.  

\begin{definition}\label{cas-def}
Given a real representation $\tau:\hg\rightarrow\End(V)$ of a compact Lie algebra $\hg$ and a bi-invariant inner product $\ip$ on $\hg$, the map 
$$
\cas_{\tau,\ip}:=-\sum_i \tau(X_i)^2:V\longrightarrow V, 
$$
is called the {\it Casimir operator} of $\tau$ with respecto to $\ip$, where $\{ X_i\}_{i=1}^{\dim{\hg}}$ is any $\ip$-orthonormal basis of $\hg$.   
\end{definition}

It is well known that $\cas_{\tau,\ip}$ is symmetric with respect to any $\tau(\hg)$-invariant inner product on $V$ and non-negative definite, $\cas_{\tau,\ip}$ commutes with every $\tau(X)$ and that $\cas_{\ad,-\kil_\hg}=I_\hg$, where $\ad:\hg\rightarrow\End(\hg)$ denotes the adjoint representation of $\hg$.  

In this section, we compute the Ricci curvature of a $G$-invariant metric $g=(x_1,x_2,x_3)_{g_b}$ as in \eqref{gg-harm} on an aligned homogeneous space $M=G/K$ with $s=2$ and positive constants $c_1,c_2$ and $\lambda_1,\dots,\lambda_t$ as in \S\ref{preli3}.  We set $\lambda_0:=0$ in order to simplify some formulas.  We first consider, for each $i=1,2$, the homogeneous space $M_i=G_i/\pi_i(K)$ with $\kil_{\ggo_i}$-orthogonal reductive decomposition $\ggo_i=\pi_i(\kg)\oplus\pg_i$ endowed with the standard metric, which will be denoted by $\gk^i$ (i.e., $\gk^i=-\kil_{\ggo_i}|_{\pg_i\times\pg_i}$).  According to \cite{WngZll2} (see also \cite{stab}), 
\begin{equation}\label{MgBi}
\Mm(\gk^i) = \unm\cas_{\chi_i} - \unc I_{\pg_i}, \qquad 
\Ricci(\gk^i) = \unm\cas_{\chi_i} + \unc I_{\pg_i}, \qquad\forall i=1,2,
\end{equation}
where 
$$
\cas_{\chi_i}:\pg_i\longrightarrow\pg_i,
$$ 
is the Casimir operator of the isotropy representation $\chi_i:\pi_i(\kg)\rightarrow\End(\pg_i)$ of $M_i=G_i/\pi_i(K)$ with respect to the bi-invariant inner product $-\kil_{\ggo_i}|_{\pi_i(\kg)\times\pi_i(\kg)}$.  Since $\{\sqrt{c_i}Z^\alpha_i\}_{\alpha=1}^{\dim{\kg}}$ is a $-\kil_{\ggo_i}$-orthonormal basis of $\pi_i(\kg)$ (see \eqref{al2}) and $\ad_{\ggo_i}{Z^\alpha_i}|_{\pg_i} = \ad_\ggo{Z^\alpha}|_{\pg_i}$, we have that 
\begin{equation}\label{Cchi}
\cas_{\chi_i} = c_i\cas_\chi|_{\pg_i}, \qquad\forall i=1,2, 
\end{equation}
where 
$$
\cas_\chi:\pg\longrightarrow\pg,
$$ 
is the Casimir operator of the isotropy representation $\chi:\kg\rightarrow\End(\pg)$ of $M=G/K$ with respect to $-\kil_{\ggo}|_{\kg\times\kg}$.

\begin{proposition}\label{ricggals2}
If $s=2$, then the Ricci operator of the metric $g=(x_1,x_2,x_3)_{g_b}$ (see \eqref{gbdef}, \eqref{dec1} and \eqref{gg-harm}) is given as follows: 
\begin{enumerate}[{\rm (i)}]
\item $\Ricci(g)|_{\pg_1} =  \tfrac{1}{4x_1z_1}I_{\pg_1} 
+ \tfrac{1}{2x_1}  \left(\tfrac{1}{z_1} - \tfrac{x_{3}}{x_1c_1B_{3}}\right) \cas_{\chi_{1}}$.  
\item[ ]
\item $\Ricci(g)|_{\pg_2} =  \tfrac{1}{4x_2z_2}I_{\pg_2} 
+ \tfrac{1}{2x_2}\left(\tfrac{1}{z_2} - \tfrac{x_{3}}{x_2c_2B_{3}} A_{3}^2\right) \cas_{\chi_{2}} $.
\item[ ]
\item The decomposition $\pg=\pg_1\oplus\pg_2\oplus\pg_3^0\oplus\dots\oplus\pg_3^t$ is $\ricci(g)$-orthogonal.
\item[ ]
\item $\Ricci(g)|_{\pg_3^l}= r_lI_{\pg_3^l}$ for any $l=0,1,\dots,t$, where 
\begin{align*}
r_l :=& \tfrac{\lambda_l}{4x_{3}B_3}\left(\tfrac{2x_1^2-x_{3}^2}{x_1^2} 
+\tfrac{(2x_2^2-x_{3}^2)A_3^2}{x_2^2}  
-\tfrac{1+A_3}{B_3}\left(\tfrac{z_1}{c_1}+\tfrac{z_2}{c_2}A_3^3\right)\right) \\ 
&+\tfrac{1}{4x_{3}B_3}\left(2\left(\tfrac{1}{c_1}+\tfrac{1}{c_2}A_3^2\right) 
-\tfrac{2x_1^2-x_{3}^2}{x_1^2c_1} 
- \tfrac{(2x_2^2-x_{3}^2)A_3^2}{x_2^2c_2}\right).
\end{align*} 
\end{enumerate}
\end{proposition}

\begin{proof} 
We first prove parts (i) and (ii).  If $X\in\pg_k$, $k=1,2$, then by \eqref{Ric2} and \eqref{mm4}, 
\begin{align*}
g(\Ricci(g)X,X) =& g(\Mm(g)X,X) - \unm\kil_\ggo(X,X) \\ 
=& -\unm\sum_{\substack{\alpha,\beta}} \tfrac{1}{x_k^2}g([X,e^k_\alpha]_\pg,e^k_\beta)^2 
+\unc\sum_{\substack{\alpha,\beta}} \tfrac{1}{x_k^2}g([e^k_\alpha,e^k_\beta]_\pg,X)^2 \\ 
& -\unm\sum_{\substack{\alpha,\beta}} \tfrac{1}{x_kx_3}g([X,e^k_\alpha]_\pg,e^3_\beta)^2 
+\unm\sum_{\substack{\alpha,\beta}} \tfrac{1}{x_kx_3}g([e^k_\alpha,e^3_\beta]_\pg,X)^2 \\ 
& -\unm\sum_{\substack{\alpha,\beta}} \tfrac{1}{x_kx_3}g([X,e^3_\beta]_\pg,e^k_\alpha)^2 
- \unm\kil_\ggo(X,X) \\ 
=& -\unm\sum_{\substack{\alpha,\beta}} g_b([X,e^k_\alpha],e^k_\beta)^2 
+\unc\sum_{\substack{\alpha,\beta}}g_b([e^k_\alpha,e^k_\beta],X)^2 
- \unm\kil_{\ggo_k}(X,X) \\ 
& -\unm\sum_{\substack{\alpha,\beta}} \tfrac{x_3}{x_k}g_b([X,e^k_\alpha],e^3_\beta)^2. 
\end{align*}
This implies that
\begin{align*}
x_kg_b(\Ricci(g)X,X) = g_b(\Ricci(g_b^k)X,X) 
-\tfrac{x_3}{2x_k} \sum_{\substack{\alpha,\beta}} g_b([X,e^k_\alpha],e^3_\beta)^2, 
\end{align*}
where $g_b^k$ is the normal metric on $M_k$ defined by $g_b|_{\ggo_k}=z_k(-\kil_{\ggo_k})$, and since $e^3_\beta=\tfrac{1}{\sqrt{B_3}}(Z^\beta_1,A_3Z^\beta_2)$ and by \eqref{Cchi},
\begin{align}
\sum_{\substack{\alpha,\beta}} g_b([X,e^k_\alpha],Z^\beta)^2 
=& -g_b\Big(\sum_{\substack{\beta}} (\ad{Z^\beta}|_{\pg_k})^2X,X\Big) = g_b(\cas_\chi|_{\pg_k}X,X) \label{chik} \\ 
=& g_b(\tfrac{1}{c_k}\cas_{\chi_k}X,X),\notag 
\end{align}
we obtain that
\begin{align*}
x_1g_b(\Ricci(g)X,X) =
& g_b(\Ricci(g_b^1)X,X) 
- \tfrac{x_3}{2x_1B_3} g_b(\tfrac{1}{c_1}\cas_{\chi_1}X,X), \\ 
x_2g_b(\Ricci(g)X,X) =
& g_b(\Ricci(g_b^2)X,X) 
-\tfrac{x_3A_3^2}{2x_2B_3} g_b(\tfrac{1}{c_2}\cas_{\chi_2}X,X). 
\end{align*}
The fact that $\Ricci(g_b^k) = \tfrac{1}{z_k}\Ricci(\gk^k) = \tfrac{1}{2z_k}\cas_{\chi_k}+\tfrac{1}{4z_k}I_{\pg_k}$ (see \eqref{MgBi} and recall that $g_b^k=z_k\gk^k$) completes the proof of parts (i) and (ii).  

We now consider $X\in\pg_3^l$ in order to prove part (iv).  We have that,
\begin{align*}
g(\Ricci(g)X,X) =& g(\Mm(g)X,X) - \unm\kil_\ggo(X,X) \\ 
=& -\unm\sum_{\substack{\alpha,\beta\\ k=1,2,3}} \tfrac{1}{x_k^2}g([X,e^k_\alpha]_\pg,e^k_\beta)^2 
+\unc\sum_{\substack{\alpha,\beta\\ k=1,2,3}} \tfrac{1}{x_k^2}g([e^k_\alpha,e^k_\beta]_\pg,X)^2 -\unm\kil_\ggo(X,X)\\ 
=& -\unm\sum_{\substack{\alpha,\beta\\ k=1,2,3}} g_b([X,e^k_\alpha],e^k_\beta)^2 
+\unc\sum_{\substack{\alpha,\beta\\ k=1,2,3}} \tfrac{x_3^2}{x_k^2}g_b([e^k_\alpha,e^k_\beta],X)^2 -\unm\kil_\ggo(X,X),
\end{align*}
and since 
$$
\sum_{\substack{\alpha,\beta}} g_b([X,e^k_\alpha],e^k_\beta)^2 = -\tr{(\ad_\pg{X}|_{\pg_k})^2}, \qquad\forall k=1,2,3,
$$
we obtain that
\begin{align*}
g(\Ricci(g)X,X) 
=& \tfrac{2x_1^2-x_3^2}{4x_1^2} \tr{(\ad{X}|_{\pg_1})^2} 
+ \tfrac{2x_2^2-x_3^2}{4x_2^2} \tr{(\ad{X}|_{\pg_2})^2} \\
&+  \unc\tr{(\ad_\pg{X}|_{\pg_3})^2} -\unm\kil_\ggo(X,X).
\end{align*}
If $X=(Z_1,A_3Z_2)$, $Z\in\kg_l$, then by \eqref{al2}, 
\begin{align*}
\kil_\ggo(X,X) =& \kil_{\ggo_1}(Z_1,Z_1)+A_3^2 \kil_{\ggo_2}(Z_2,Z_2) 
= \tfrac{1}{c_1}\kil_{\ggo}(Z,Z)+\tfrac{1}{c_2}A_3^2 \kil_{\ggo}(Z,Z) \\
=& -\left(\tfrac{1}{c_1}+\tfrac{1}{c_2}A_3^2\right) \la Z,Z\ra, 
\end{align*}
by \eqref{cij},
\begin{align*}
\tr{(\ad{X}|_{\pg_1})^2} =& \tr{(\ad{Z_1}|_{\pg_1})^2}=\kil_{\ggo_1}(Z_1,Z_1)-\kil_{\pi_1(\kg)}(Z_1,Z_1)\\ 
=& (1-\lambda_lc_1)\kil_{\ggo_1}(Z_1,Z_1) = -\tfrac{1-\lambda_lc_1}{c_1}\la Z,Z\ra, 
\end{align*}
\begin{align*}
\tr{(\ad{X}|_{\pg_2})^2} =& A_3^2\tr{(\ad{Z_2}|_{\pg_2})^2} = -\tfrac{1-\lambda_lc_2}{c_2}A_3^2\la Z,Z\ra, 
\end{align*}
and by \eqref{gbpi},
$$
g_b(X,X) = g_b(Z_1,Z_1)+A_3^2g_b(Z_2,Z_2) = \left(\tfrac{z_1}{c_1}+\tfrac{z_2}{c_2}A_3^2\right)\la Z,Z\ra = B_3\la Z,Z\ra.  
$$
On the other hand, using that the $g_b$-orthogonal projections on $\kg$ and $\pg_3$ of a vector $(Y_1,W_2)\in\pi_1(\kg)\oplus\pi_2(\kg)$, $Y,W\in\kg$ (see \eqref{pi1pi2}) are respectively given by $U$ and $(V_1,A_3V_2)$, where
$$
U=-\tfrac{A_3}{1-A_3}Y+\tfrac{1}{1-A_3}W, \qquad V=\tfrac{1}{1-A_3}Y-\tfrac{1}{1-A_3}W,
$$
we obtain that 
$$
[X,e^3_\alpha]_{\pg} = [X,e^3_\alpha]_{\pg_3} 
= \tfrac{1}{\sqrt{B_3}}\left([Z,Z^\alpha]_1,A_3^2[Z,Z^\alpha]_2\right)_{\pg_3} 
= \tfrac{1+A_3}{\sqrt{B_3}}\left([Z_1,Z^\alpha_1],A_3[Z_2,Z^\alpha_2]\right),
$$
by setting $Y:=[Z,Z^\alpha]$ and $W:=A_3^2[Z,Z^\alpha]$.  Thus by \eqref{al1},
\begin{align*}
\tr{(\ad_\pg{X}|_{\pg_3})^2} =& 
\sum_{\substack{\alpha}} g_b((\ad_\pg{X}|_{\pg_3})^2e^3_\alpha,e^3_\alpha) \\ 
=&\sum_{\substack{\alpha}} g_b\left(\ad_\pg{X}|_{\pg_3}\tfrac{1+A_3}{\sqrt{B_3}}\left([Z_1,Z^\alpha_1],A_3[Z_2,Z^\alpha_2]\right), \tfrac{1}{\sqrt{B_3}}(Z^\alpha_1,A_3Z^\alpha_2)\right) \\
=& \tfrac{1+A_3}{B_3}\sum_{\substack{\alpha}} g_b\left((\ad{Z_1})^2Z^\alpha_1,A_3^2(\ad{Z_2})^2Z^\alpha_2), (Z^\alpha_1,A_3Z^\alpha_2)\right) \\
=& \tfrac{1+A_3}{B_3}\sum_{\substack{\alpha}} g_b\left((\ad{Z_1})^2Z^\alpha_1,Z^\alpha_1\right) 
+g_b\left(A_3^2(\ad{Z_2})^2Z^\alpha_2, A_3Z^\alpha_2\right) \\ 
=& \tfrac{1+A_3}{B_3}\sum_{\substack{\alpha}} \tfrac{z_1}{c_1}\la(\ad{Z})^2Z^\alpha,Z^\alpha\ra 
+\tfrac{z_2}{c_2}A_3^3\la(\ad{Z})^2Z^\alpha, Z^\alpha\ra \\
=& \tfrac{1+A_3}{B_3}\left(\tfrac{z_1}{c_1}+\tfrac{z_2}{c_2}A_3^3\right)\kil_{\kg_l}(Z,Z) 
= -\tfrac{1+A_3}{B_3}\left(\tfrac{z_1}{c_1}+\tfrac{z_2}{c_2}A_3^3\right)\lambda_l\la Z,Z\ra.
\end{align*}
All this implies that
\begin{align*}
x_3g_b(\Ricci(g)X,X) 
=& -\tfrac{2x_1^2-x_3^2}{4x_1^2} \tfrac{1-\lambda_lc_1}{c_1}\la Z,Z\ra 
- \tfrac{2x_2^2-x_3^2}{4x_2^2}A_3^2 \tfrac{1-\lambda_lc_2}{c_2}\la Z,Z\ra \\
&-  \tfrac{1+A_3}{4B_3}\left(\tfrac{z_1}{c_1}+\tfrac{z_2}{c_2}A_3^3\right)\lambda_l \la Z,Z\ra
+ \tfrac{1}{2B_3}\left(\tfrac{1}{c_1}+\tfrac{1}{c_2}A_3^2\right) g_b(X,X)\\ 
=& -\tfrac{(2x_1^2-x_3^2)(1-\lambda_lc_1)}{4x_1^2c_1B_3} g_b(X,X) 
- \tfrac{(2x_2^2-x_3^2)(1-\lambda_lc_2)A_3^2}{4x_2^2c_2B_3} g_b(X,X) \\
&-  \tfrac{1+A_3}{4B_3^2}\left(\tfrac{z_1}{c_1}+\tfrac{z_2}{c_2}A_3^3\right)\lambda_l g_b(X,X)
+ \tfrac{1}{2B_3}\left(\tfrac{1}{c_1}+\tfrac{1}{c_2}A_3^2\right) g_b(X,X),
\end{align*}
and so the formula for $r_l$ given in (iv) follows.  

Finally, to prove (iii), we take $X\in\pg_k$ and $Y\in\pg_l$, $k< l$, and compute
\begin{align*}
g(\Ricci(g)X,Y) =& g(\Mm(g)X,Y) - \unm\kil_\ggo(X,Y) \\ 
=& -\unm\sum_{\substack{\alpha,\beta\\ i,j}} \tfrac{x_j}{x_i}g_b([X,e^i_\alpha]_\pg,e^j_\beta) g_b([Y,e^i_\alpha]_\pg,e^j_\beta) \\
& +\unc\sum_{\substack{\alpha,\beta\\ i,j}} \tfrac{x_kx_l}{x_ix_j}g_b([e^i_\alpha,e^j_\beta]_\pg,X) g_b([e^i_\alpha,e^j_\beta]_\pg,Y) 
- \unm\kil_\ggo(X,Y). 
\end{align*}
Since $\kil_\ggo(X,Y)=\kil_{\ggo_k}(X,Y_k)=0$ and  
$$
\sum_{\substack{\alpha,\beta}} g([X,e^i_\alpha]_\pg,e^j_\beta)g([Y,e^i_\alpha]_\pg,e^j_\beta) 
= -\tr{\ad_{\pg_k}{X}\ad{Y}|_{\pg_k}} = -\kil_{\ggo_k}(X,Y_k)=0,   
$$
if $k=1,2$ and $l=3$, and zero otherwise, we obtain that $g(\Ricci(g)X,Y)=0$, concluding the proof. 
\end{proof}

\section{The symmetric form $(H_Q)_g^2$}\label{HQ2-sec}

In order to study the BRF condition \eqref{BRF-def} for the generalized metric $(g,H_Q)$ on an aligned space $M=G/K$ with $s=2$, where $g=(x_1,x_2,x_3)_{g_b}$ is as in \eqref{gg-harm} and $H_Q$ is as in \eqref{HQ-def}, we compute in this section the symmetric bilinear form $(H_Q)_g^2$.   

\begin{lemma}\label{HQ2ort}
\hspace{1cm}
\begin{enumerate}[{\rm (i)}]
\item $(H_Q)^2_g(\pg_1,\pg_2)=0$. 

\item $(H_Q)^2_g(\pg_1,\pg_3)=0$ and $(H_Q)^2_g(\pg_2,\pg_3)=0$. 

\item $(H_Q)^2_g(\pg_3^l,\pg_3^m)=0$ for all $0\leq l\ne m\leq t$. 
\end{enumerate}
\end{lemma}

\begin{proof}
According to Remark \ref{HQnz}, if $i\ne j$ then
$$
H_Q(\pg_i,X,Y)H_Q(\pg_j,X,Y)=0, \qquad\forall X\in\pg_k, \quad Y\in\pg_l, \quad 1\leq k,l\leq 3,
$$
from which parts (i) and (ii) follow.  For part (iii) recall that the irreducible $\Ad(K)$-components of $\pg_3^l$'s are pairwise non-equivalent.  
\end{proof}

We now give a formula for the form $(H_Q)^2_g$ on $\pg_1$ and $\pg_2$. 

\begin{proposition}\label{HQ2s} 
For any $X\in\pg_k$, $k=1,2$,
$$
(H_Q)_g^2(X,X) = g_b\left(\left(\left(\tfrac{2S_k}{x_kc_k}-\tfrac{2y_k^2}{x_k^2z_k^3}\right)\cas_{\chi_k}+\tfrac{y_k^2}{x_k^2z_k^3}I_{\pg_k}\right)X,X\right),
$$
where
$$
S_1= \tfrac{1}{x_3B_3}\left(\tfrac{y_1}{z_1} + \tfrac{C_3}{B_4}\right)^2, \qquad 
S_2 = \tfrac{1}{x_3B_3}\left(A_3\tfrac{y_2}{z_2}+ \tfrac{C_3}{B_4}\right)^2.
$$
\end{proposition}

\begin{proof} 
Using \eqref{pipj0}-\eqref{pipj3} and $Q(\kg,\pg_k)=0$ one obtains that,
\begin{align*}
(H_Q)_g^2(X,X) 
=& \sum_{\substack{\alpha,\beta\\ 1\leq i,j\leq 3}} \tfrac{1}{x_ix_j} \left( Q([X,e^i_\alpha],e^j_\beta) + Q([X,e^i_\alpha]_\kg,e^j_\beta) - Q([X,e^j_\beta]_\kg,e^i_\alpha)\right)^2\\
=& \sum_{\substack{\alpha,\beta}} \tfrac{1}{x_k^2} \left( Q([X,e^k_\alpha],e^k_\beta) + Q([X,e^k_\alpha]_\kg,e^k_\beta) - Q([X,e^k_\beta]_\kg,e^k_\alpha)\right)^2\\ 
& + \sum_{\substack{\alpha,\beta}} \tfrac{1}{x_kx_3} \left( Q([X,e^k_\alpha],e^3_\beta) + Q([X,e^k_\alpha]_\kg,e^3_\beta) - Q([X,e^3_\beta]_\kg,e^k_\alpha)\right)^2\\
& + \sum_{\substack{\alpha,\beta}} \tfrac{1}{x_3x_k} \left( Q([X,e^3_\alpha],e^k_\beta) + Q([X,e^3_\alpha]_\kg,e^k_\beta) - Q([X,e^k_\beta]_\kg,e^3_\alpha)\right)^2\\
=& \tfrac{1}{x_k^2}\sum_{\substack{\alpha,\beta}} Q([X,e^k_\alpha],e^k_\beta)^2 
+ \tfrac{2}{x_kx_3}\sum_{\substack{\alpha,\beta}}  \left( Q([X,e^k_\alpha],e^3_\beta) + Q([X,e^k_\alpha]_\kg,e^3_\beta)\right)^2.  
\end{align*}
It follows from  \eqref{pipj2} and \eqref{MgBi} that,
\begin{align*}
\sum_{\substack{\alpha,\beta}} Q([X,e^k_\alpha],e^k_\beta)^2 
=& y_k^2\sum_{\substack{\alpha,\beta}} \gk^k([X,e^k_\alpha],e^k_\beta)^2 
= -\tfrac{y_k^2}{z_k^2}\tr{(\ad_{\pg_k}{X})^2} \\ 
=& -\tfrac{4y_k^2}{z_k^2}\gk^k(\Mm(\gk^k)X,X) = -\tfrac{4y_k^2}{z_k^2}\gk^k((\unm\cas_{\chi_k}-\unc I_{\pg_k})X,X) \\ 
=& -\tfrac{4y_k^2}{z_k^3}g_b((\unm\cas_{\chi_k}-\unc I_{\pg_k})X,X).
\end{align*}
On the other hand, by \eqref{L71},  
$$
Q([X,e^k_\alpha]_\kg,e^3_\beta) = g_b([X,e^k_\alpha]_\kg,e^4_\beta)Q(e^4_\beta,e^3_\beta) = \tfrac{1}{B_4}g_b([X,e^k_\alpha],Z^\beta)\tfrac{-C_3}{\sqrt{B_3}},
$$
and using \eqref{basis} we obtain that, 
\begin{align*}
Q([X,e^k_\alpha],e^3_\beta) =& -\tfrac{y_k}{z_k}g_b([e^3_\beta,X],e^k_\alpha) = -\tfrac{y_k}{z_k\sqrt{B_3}}g_b([\vp_3(Z^\beta),X],e^k_\alpha) \\ 
=& \left\{\begin{array}{l} 
-\tfrac{y_1}{z_1\sqrt{B_3}}g_b([Z^\beta,X],e^1_\alpha), \qquad k=1,\\ \\
-\tfrac{y_2A_3}{z_2\sqrt{B_3}}g_b([Z^\beta,X],e^2_\alpha), \qquad k=2.
\end{array}\right.
\end{align*}
All this implies that
\begin{align*}
(H_Q)_g^2(X,X)  
=& g_b\left(\left(-\tfrac{2y_1^2}{x_1^2z_1^3}\cas_{\chi_1}+\tfrac{y_1^2}{x_1^2z_1^3}I_{\pg_1}\right)X,X\right) \\
&+ \tfrac{2}{x_1x_3} \left(\tfrac{y_1}{z_1\sqrt{B_3}} + \tfrac{C_3}{B_4\sqrt{B_3}}\right)^2 \sum_{\substack{\alpha,\beta}}g_b([X,e^1_\alpha],Z^\beta)^2, \qquad\forall X\in\pg_1,
\end{align*}
and
\begin{align*}
(H_Q)_g^2(X,X) 
=& g_b\left(\left(-\tfrac{2y_2^2}{x_2^2z_2^3}\cas_{\chi_2}+\tfrac{y_2^2}{x_2^2z_2^3}I_{\pg_2}\right)X,X\right) \\ 
&+ \tfrac{2}{x_2x_3}\left(\tfrac{y_2A_3}{z_k\sqrt{B_3}} + \tfrac{C_3}{B_4\sqrt{B_3}}\right)^2 \sum_{\substack{\alpha,\beta}}g_b([X,e^2_\alpha],Z^\beta)^2, \qquad\forall X\in\pg_2.
\end{align*}
Finally, formula \eqref{chik} concludes the proof if we set $S_k$ as in the proposition. 
\end{proof}

On the other hand, we have for each fixed $e^3_\gamma\in\pg_3$ that 
\begin{align}
(H_Q)_{g}^2(e^3_\gamma,e^3_\gamma) 
= \sum_{\substack{\alpha,\beta\\ i=1,2}} \tfrac{1}{x_i^2}H_Q(e^3_\gamma,e^i_\alpha,e^i_\beta)^2 
+ \sum_{\substack{\alpha,\beta}} \tfrac{1}{x_3^2}H_Q(e^3_\gamma,e^3_\alpha,e^3_\beta)^2. \label{HQ2}
\end{align} 

\begin{lemma}\label{HQ2-1}
For any $e^3_\gamma\in\pg_3^l$, where $l=0,1,\dots,t$, 
\begin{align*}
\sum_{\substack{\alpha,\beta\\ i=1,2}}  \tfrac{1}{x_i^2}H_Q(e^3_\gamma,e^i_\alpha,e^i_\beta)^2 
= & \tfrac{1}{B_3x_1^2}\left(\tfrac{y_1}{z_1} + \tfrac{C_3}{B_4} \right)^2 \tfrac{1-c_1\lambda_l}{c_1} 
+ \tfrac{1}{B_3x_2^2}\left(\tfrac{y_2A_3}{z_2} + \tfrac{C_3}{B_4}\right)^2 \tfrac{1-c_2\lambda_l}{c_2}.
\end{align*} 
\end{lemma}

\begin{proof}
Since $Q(\kg,\pg_i)=0$, for $i=1,2$ we have that 
\begin{align*}
H_Q(e^3_\gamma,e^i_\alpha,e^i_\beta) 
= Q([e^3_\gamma,e^i_\alpha],e^i_\beta) + Q([e^i_\alpha,e^i_\beta]_\kg,e^3_\gamma),
\end{align*} 
and since
\begin{align*}
Q([e^3_\gamma,e^i_\alpha],e^i_\beta) 
=& -\tfrac{y_i}{z_i}g_b([e^3_\gamma,e^i_\alpha],e^i_\beta) 
= \left\{\begin{array}{cl} 
-\tfrac{y_1}{z_1\sqrt{B_3}}g_b([Z^\gamma,e^1_\alpha],e^1_\beta), & \qquad i=1, \\ \\ 
-\tfrac{y_2A_3}{z_2\sqrt{B_3}}g_b([Z^\gamma,e^2_\alpha],e^2_\beta), & \qquad i=2, 
\end{array}\right.
\end{align*} 
and
\begin{align*}
Q([e^i_\alpha,e^i_\beta]_\kg,e^3_\gamma) 
=& g_b([e^i_\alpha,e^i_\beta]_\kg,e^4_\gamma) Q(e^4_\gamma,e^3_\gamma) 
= g_b([e^4_\gamma,e^i_\alpha],e^i_\beta) \tfrac{-C_3}{\sqrt{B_4B_3}} \\
=& \tfrac{-C_3}{B_4\sqrt{B_3}} g_b([Z^\gamma,e^i_\alpha],e^i_\beta), 
\end{align*} 
one obtains that
\begin{align*}
\sum_{\substack{\alpha,\beta\\ i=1,2}} \tfrac{1}{x_i^2}H_Q(e^3_\gamma,e^i_\alpha,e^i_\beta)^2 
=&  -\tfrac{1}{B_3x_1^2}\left(\tfrac{y_1}{z_1} + \tfrac{C_3}{B_4} \right)^2 \tr{(\ad{Z^\gamma}|_{\pg_1})^2} \\ 
& -\tfrac{1}{B_3x_2^2}\left(\tfrac{y_2A_3}{z_2} + \tfrac{C_3}{B_4}\right)^2 \tr{(\ad{Z^\gamma}|_{\pg_2})^2}. 
\end{align*} 
On the other hand, if in addition $e^3_\gamma\in\pg_k^l$, i.e., $Z^\gamma\in\kg_l$, then by \eqref{al2} and \eqref{cij},
\begin{align*}
\tr{(\ad{Z^\gamma}|_{\pg_i})^2} 
=& \tr{(\ad_{\ggo_i}{Z^\gamma_i})^2} - \tr{(\ad{Z^\gamma_i}|_{\pi_i(\kg)})^2} 
= \kil_{\ggo_i}(Z^\gamma_i,Z^\gamma_i) - \kil_{\pi_i(\kg)}(Z^\gamma_i,Z^\gamma_i) \\ 
=& (1-c_i\lambda_l)\kil_{\ggo_i}(Z^\gamma_i,Z^\gamma_i)  
= (1-c_i\lambda_l)\tfrac{-1}{c_i} = -\tfrac{1-c_i\lambda_l}{c_i}, \qquad i=1,2,
\end{align*}
concluding the proof.  
\end{proof}

\begin{proposition}\label{HQ2-s2}
If $e^3_\gamma\in\pg_3^l$, where $l=0,1,\dots,t$, then 
\begin{align*}
(H_Q)_{g}^2(e^3_\gamma,e^3_\gamma) 
= & \tfrac{1}{x_1^2B_3} \left(\tfrac{y_1}{z_1} + \tfrac{C_3}{B_{4}} \right)^2 \tfrac{1-c_1\lambda_l}{c_1} 
 +\tfrac{1}{x_2^2B_3} \left(\tfrac{y_2A_3}{z_2} + \tfrac{C_3}{B_{4}}\right)^2 \tfrac{1-c_2\lambda_l}{c_2} \\ 
& +\tfrac{\lambda_l}{x_3^2B_3^3}\left(\tfrac{y_1}{c_1}+A_3^3\tfrac{y_2}{c_2} + \tfrac{3C_3}{B_{4}\sqrt{B_4}} \left(\tfrac{z_1}{c_1}+A_3^2\tfrac{z_2}{c_2}\right)\right)^2.  
\end{align*}
\end{proposition}

\begin{proof} 
It follows from \eqref{HQ2} and Lemma \ref{HQ2-1} that 
\begin{align*}
(H_Q)_{g}^2(e^3_\gamma,e^3_\gamma) 
= & \tfrac{1}{x_1^2B_3} \left(\tfrac{y_1}{z_1} + \tfrac{C_3}{B_{4}} \right)^2 \tfrac{1-c_1\lambda_l}{c_1} 
 +\tfrac{1}{x_2^2B_3} \left(\tfrac{y_2A_3}{z_2} + \tfrac{C_3}{B_{4}}\right)^2 \tfrac{1-c_2\lambda_l}{c_2}\\ 
 &+ \tfrac{1}{x_3^2}\sum_{\substack{\alpha,\beta}} H_Q(e^3_\gamma,e^3_\alpha,e^3_\beta)^2.
\end{align*}
By \eqref{gbpi},
\begin{align*}
Q([e^3_\gamma,e^3_\alpha],e^3_\beta) 
=& Q(\tfrac{1}{B_3}([Z^\gamma_1,Z^\alpha_1],A_3^2[Z^\gamma_2,Z^\alpha_2]),\tfrac{1}{\sqrt{B_3}}(Z^\beta_1,A_3Z^\beta_2)) \\ 
=& \tfrac{1}{B_3\sqrt{B_3}}\left(-\tfrac{y_1}{z_1}g_b([Z^\gamma_1,Z^\alpha_1],Z^\beta_1) 
-\tfrac{y_2}{z_2}A_3^3g_b([Z^\gamma_2,Z^\alpha_2],Z^\beta_2)\right) \\ 
=& \tfrac{-1}{B_3\sqrt{B_3}}\left(\tfrac{y_1}{c_1}+A_3^3\tfrac{y_2}{c_2}\right)\la[Z^\gamma,Z^\alpha],Z^\beta\ra,   
\end{align*} 
and by \eqref{L71},
\begin{align*}
Q([e^3_\gamma,e^3_\alpha]_\kg,e^3_\beta)
=& g_b([e^3_\gamma,e^3_\alpha],e^4_\beta)Q(e^4_\beta,e^3_\beta)
=\tfrac{-C_3}{B_{4}\sqrt{B_3}} g_b([e^{4}_\beta,e^3_\gamma],e^3_\alpha) \\ 
=&\tfrac{-C_3}{B_{4}\sqrt{B_3}} g_b([\tfrac{1}{\sqrt{B_4}}(Z^\beta_1,Z^\beta_2),\tfrac{1}{\sqrt{B_3}}(Z^\gamma_1,A_3Z^\gamma_2)],\tfrac{1}{\sqrt{B_3}}(Z^\alpha_1,A_3Z^\alpha_2)) \\
=&\tfrac{-C_3}{B_{4}\sqrt{B_4}B_3\sqrt{B_3}} \left(g_b([Z^\beta_1,Z^\gamma_1],Z^\alpha_1) +A_3^2g_b([Z^\beta_2,Z^\gamma_2],Z^\alpha_2)\right) \\ 
=&\tfrac{-C_3}{B_{4}\sqrt{B_4}B_3\sqrt{B_3}} \left(\tfrac{z_1}{c_1}+A_3^2\tfrac{z_2}{c_2}\right)\la[Z^\gamma,Z^\alpha],Z^\beta\ra,  
\end{align*} 
which also gives 
$$
Q([e^3_\gamma,e^3_\alpha]_\kg,e^3_\beta) = -Q([e^3_\gamma,e^3_\beta]_\kg,e^3_\alpha) = Q([e^3_\alpha,e^3_\beta]_\kg,e^3_\gamma).
$$
All this implies that 
\begin{align*}
&\sum_{\substack{\alpha,\beta}} H_Q(e^3_\gamma,e^3_\alpha,e^3_\beta)^2 \\
=& \sum_{\substack{\alpha,\beta}} \left(\tfrac{-1}{B_3\sqrt{B_3}}\left(\tfrac{y_1}{c_1}+A_3^3\tfrac{y_2}{c_2}\right) + \tfrac{-3C_3}{B_{4}\sqrt{B_4}B_3\sqrt{B_3}} \left(\tfrac{z_1}{c_1}+A_3^2\tfrac{z_2}{c_2}\right)\right)^2 \la[Z^\gamma,Z^\alpha],Z^\beta\ra^2  \\ 
=& \tfrac{1}{B_3^3}\left(\tfrac{y_1}{c_1}+A_3^3\tfrac{y_2}{c_2} + \tfrac{3C_3}{B_{4}\sqrt{B_4}} \left(\tfrac{z_1}{c_1}+A_3^2\tfrac{z_2}{c_2}\right)\right)^2 (-\kil_\kg(Z^\gamma,Z^\gamma)),  
\end{align*} 
and since by \eqref{al1},
$$
\kil_\kg(Z^\gamma,Z^\gamma) = \kil_{\kg_l}(Z^\gamma,Z^\gamma) = \lambda_l\kil_\ggo(Z^\gamma,Z^\gamma) = -\lambda_l,
$$
the lemma follows.  
\end{proof}

\section{Bismut Ricci flat metrics}\label{BRF-sec}

In this section, we study the existence of BRF generalized metrics on an aligned homogeneous space $M=G/K$ with $s=2$.  Recall from \S\ref{preli3} the positive constants $c_1,c_2$ and $\lambda_1,\dots,\lambda_t$,  $\lambda_0=0$, and that 
$$
c_2=\tfrac{c_1}{c_1-1}, \qquad y_2=-\tfrac{y_1}{c_1-1}.
$$
Without any lost of generality, we can consider the following normalizations (see \S\ref{preli1}): 
$$
1<c_1\leq 2\leq c_2,  \qquad  z_2=\tfrac{1}{c_1-1}, \qquad y_1=1,
$$
from which it is easy to check that 
$$
A_3=-z_1, \qquad 
B_3= \tfrac{z_1(z_1+1)}{c_1}, 
\qquad
B_{4}=\tfrac{z_1+1}{c_1}, \qquad
C_3= \tfrac{z_1+1}{c_1}, 
$$
and the constants involved in Proposition \ref{HQ2s} are given by 
\begin{align}
S_{1}=\tfrac{c_1(z_1+1)}{x_{3}z_1^3}, \qquad
S_{2} = \tfrac{c_1(z_1+1)}{x_{3}z_1}. \label{S1S2}
\end{align}

We are finally ready to prove the main result of this paper.  

\begin{theorem}\label{BRFggs2}
Let $M=G/K$ be an aligned homogeneous space with $s=2$ such that Assumption \ref{assum} holds.   
\begin{enumerate}[{\rm (i)}]
\item The $G$-invariant generalized metric $(g_0,H_{0})$ defined by $H_{0}:=H_{Q_0}$, 
\begin{align}
g_0:=\left(\tfrac{1}{z_1},1,\tfrac{z_1+1}{z_1}\right)_{g_b}, \quad g_b:=z_1(-\kil_{\ggo_1})+ \tfrac{1}{c_1-1}(-\kil_{\ggo_2}), \quad Q_0:=\kil_{\ggo_1}-\tfrac{1}{c_1-1}\kil_{\ggo_2}, \label{gthm}
\end{align} 
where $z_1=c_1-1$, 
is Bismut Ricci flat.  In the case when $\kg$ is abelian, $(g_0,H_{0})$ is Bismut Ricci flat for any $z_1>0$.

\item If $\kg$ is not semisimple, then $(g_0,H_{0})$ is the unique $G$-invariant Bismut Ricci flat generalized metric on $M=G/K$ up to scaling of the form $(g=(x_1,x_2,x_3)_{g_b},H_Q)$.   

\item If $\kg$ is not abelian and neither $\cas_{\chi_{1}}$ nor $\cas_{\chi_{2}}$ is a multiple of the identity, then $(g_0,H_{0})$ is the unique $G$-invariant Bismut Ricci flat generalized metric on $M=G/K$ up to scaling of the form $(g=(x_1,x_2,x_3)_{g_b},H_Q)$. 
\end{enumerate}
\end{theorem}

\begin{remark}\label{lowdim}
The lowest dimensional (irreducible) examples provided by the theorem are:
\begin{enumerate}[\small{$\bullet$}]
\item $M^5=\SU(2)\times\SU(2)/S^1$ (found in \cite{PdsRff1}, see Example \ref{PR1-exa} below), 

\item $M^{10}=\SU(2)\times\SU(3)/S^1$ (new),  

\item $M^{12}=\SU(3)\times\SU(3)/\U(2)$ (found in \cite{PdsRff2}, see Example \ref{PR2-exa} below).  
\end{enumerate}
\end{remark}

\begin{remark}\label{rem0}
The closed $3$-form $H_0$ also depends on $z_1$; indeed, its definition \eqref{HQ-def} uses the $g_b$-orthogonal reductive decomposition.   
\end{remark}

\begin{remark}\label{rem1}
It follows from Proposition \ref{ricggals2} (see also the computation of $r_l$ in the proof below) that for any $z_1>0$, the Ricci eigenvalues of the metric $g_0$ given in \eqref{gthm} are 
$$
\unc, \quad \tfrac{c_1-1}{4} \quad\mbox{and} \quad 
\tfrac{c_1}{4}\left(1-\lambda_l\right), 
\quad l=0,\dots,t, 
$$
with multiplicities $\dim{\pg_1}$, $\dim{\pg_2}$ and $\dim{\kg_l}$, respectively.  Thus $g_0$ is Einstein if and only if $c_1=2$, $\kg$ is semisimple and $\lambda_1=\dots=\lambda_t=\unm$, which is equivalent to $M=G\times G/\Delta G$, where $G$ is a compact semisimple Lie group and $g_0=2\gk$, the symmetric metric.  
\end{remark}

\begin{remark}\label{rem4}
If $\kg$ is not abelian, then the only way to obtain $g_b=\gk$ up to scaling is when $c_1=2$ (i.e., $g_0=(1,1,2)_{\gk}$ and $Q_0=\kil_{\ggo_1}-\kil_{\ggo_2}$), e.g.\ when $M=G\times G/\Delta K$ for any non-abelian subgroup $K\subset G$.  On the other hand, $g_0$ is never a normal metric.  
\end{remark}

\begin{proof}[Proof of Theorem \ref{BRFggs2}]
It follows from Proposition \ref{harms2}, Proposition \ref{ricggals2}, (iii), (iv) and Lemma \ref{HQ2ort} that the BRF condition for $(g,H_Q)$ consists of the following three equations, one for each $\pg_k$, $k=1,2,3$:   
\begin{align}
4x_1\Ricci(g)|_{\pg_1} =& T_1, \notag\\   
4x_2\Ricci(g)|_{\pg_2} =& T_2, \notag\\
4x_3r_l  =& (H_Q)_g^2(e^3_\gamma,e^3_\gamma), \quad\forall e^3_\gamma\in\pg^l_3, \quad l=0,1,\dots,t,\label{BRF3ra}
\end{align}
where $(H_Q)_g^2|_{\pg_k\times\pg_k}=g_b(T_k\cdot,\cdot)$ for $k=1,2$; indeed, 
$$
4x_3r_l  = 4x_3g_b(\Ricci(g)e^3_\gamma,e^3_\gamma) = 4g(\Ricci(g)e^3_\gamma,e^3_\gamma) = (H_Q)_g^2(e^3_\gamma,e^3_\gamma).  
$$  
According to Proposition \ref{ricggals2}, (i), (ii) and Proposition \ref{HQ2s}, the above first and second equations are respectively equivalent to,
\begin{align*}
\tfrac{1}{z_1}I_{\pg_1} 
+ 2  \left(\tfrac{1}{z_1} - \tfrac{x_{3}}{x_1c_1B_{3}}\right) \cas_{\chi_{1}} 
=& \tfrac{y_1^2}{x_1^2z_1^3}I_{\pg_1} 
+\left(\tfrac{2S_{1}}{x_1c_1}-\tfrac{2y_1^2}{x_1^2z_1^3}\right)\cas_{\chi_{1}}, \\
\tfrac{1}{z_2}I_{\pg_2} 
+ 2 \left(\tfrac{1}{z_2} - \tfrac{x_{3}A_{3}^2}{x_2c_2B_{3}} \right) \cas_{\chi_{2}} 
=&\tfrac{y_2^2}{x_2^2z_2^3}I_{\pg_2} 
+\left(\tfrac{2S_{2}}{x_2c_2}-\tfrac{2y_2^2}{x_2^2z_2^3}\right)\cas_{\chi_{2}}, 
\end{align*}
which are easily seen to be equivalent to,
\begin{align}
I_{\pg_1} + 2  \left(1 - \tfrac{x_{3}}{x_1(z_1+1)}\right)\cas_{\chi_{1}} 
=& \tfrac{1}{x_1^2z_1^2}\left(I_{\pg_1}
+2\left(\tfrac{x_1(z_1+1)}{x_3}-1\right)\cas_{\chi_{1}}\right), \label{BRF1ra} \\
I_{\pg_2} + 2 \left(1 - \tfrac{x_{3}z_1}{x_2(z_1+1)} \right) \cas_{\chi_{2}} 
=& \tfrac{1}{x_2^2}\left(I_{\pg_2} 
+2\left(\tfrac{x_2(z_1+1)}{x_3z_1}-1\right)\cas_{\chi_{2}}\right).   \label{BRF2da}
\end{align}
In particular, if $\cas_{\chi_{1}}$ is not a multiple of the identity then $x_1=\tfrac{1}{z_1}$, and if $\cas_{\chi_{2}}$ is not a multiple of the identity then $x_2=1$.     

\begin{lemma}\label{lemma}
The third set of equations \eqref{BRF3ra} is equivalent to 
\begin{align}
& \lambda_l\left((z_1+1)^2
+ \left(\tfrac{(z_1+1)^2}{z_1^2}-x_3^2\right)\left(\tfrac{1}{x_1^2} + \tfrac{z_1^2}{x_2^2}\right) 
-\tfrac{1}{x_3^2z_1^2}\left(z_1^2-z_1+1 + \tfrac{3\sqrt{c_1}}{\sqrt{z_1+1}} z_1\right)^2\right) \label{BRF3ra-2}\\
&=\tfrac{1}{c_1}\left(\tfrac{(z_1+1)^2}{z_1^2}-x_3^2\right)\left(\tfrac{1}{x_1^2} + \tfrac{z_1^2(c_1-1)}{x_2^2}\right), \qquad\forall l=0,1,\dots,t. \notag
\end{align}
\end{lemma}

\begin{proof}[Proof of Lemma \ref{lemma}]
By Proposition \ref{ricggals2}, (iv), 
\begin{align*}
r_l =& \tfrac{\lambda_l}{4x_3B_3}\left(\tfrac{2x_1^2-x_3^2}{x_1^2} 
+\tfrac{(2x_2^2-x_3^2)A_3^2}{x_2^2}  
-\tfrac{1+A_3}{B_3}\left(\tfrac{z_1}{c_1}+\tfrac{z_2}{c_2}A_3^3\right)\right) \\ 
&+\tfrac{1}{4x_3B_3}\left(2\left(\tfrac{1}{c_1}+\tfrac{1}{c_2}A_3^2\right) 
-\tfrac{2x_1^2-x_3^2}{x_1^2c_1} 
- \tfrac{(2x_2^2-x_3^2)A_3^2}{x_2^2c_2}\right) \\
=& \tfrac{\lambda_l}{4x_3B_3}\left(\tfrac{2x_1^2-x_3^2}{x_1^2} 
+\tfrac{(2x_2^2-x_3^2)z_1^2}{x_2^2}  
-\tfrac{c_1(1-z_1)}{z_1(z_1+1)}\left(\tfrac{z_1}{c_1}-\tfrac{1}{c_1}z_1^3\right)\right) \\ 
&+\tfrac{1}{4x_3B_3}\left(2\left(\tfrac{1}{c_1}+\tfrac{c_1-1}{c_1}z_1^2\right) 
-\tfrac{2x_1^2-x_3^2}{x_1^2c_1} 
- \tfrac{(2x_2^2-x_3^2)z_1^2(c_1-1)}{x_2^2c_1}\right) \\
=& \tfrac{\lambda_l}{4x_3B_3}\left(\tfrac{2x_1^2-x_3^2}{x_1^2} 
+\tfrac{(2x_2^2-x_3^2)z_1^2}{x_2^2}  
-\tfrac{1-z_1}{z_1(z_1+1)}\left(z_1-z_1^3\right)\right) \\ 
&+\tfrac{1}{4x_3B_3c_1}\left(2\left(1+(c_1-1)z_1^2\right) 
-\tfrac{2x_1^2-x_3^2}{x_1^2} 
- \tfrac{(2x_2^2-x_3^2)z_1^2(c_1-1)}{x_2^2}\right) \\
=& \tfrac{\lambda_l}{4x_3B_3}\left(2-\tfrac{x_3^2}{x_1^2} 
+2z_1^2-\tfrac{x_3^2z_1^2}{x_2^2}  
-(z_1^2-2z_1+1)\right) \\ 
&+\tfrac{1}{4x_3B_3c_1}\left(2\left(1+(c_1-1)z_1^2\right) 
-2+\tfrac{x_3^2}{x_1^2} 
-2z_1^2(c_1-1)+ \tfrac{x_3^2z_1^2(c_1-1)}{x_2^2}\right) \\
=& \tfrac{\lambda_l}{4x_3B_3}\left((z_1+1)^2-\tfrac{x_3^2}{x_1^2} -\tfrac{x_3^2z_1^2}{x_2^2}\right) 
+\tfrac{1}{4x_3B_3c_1}\left(\tfrac{x_3^2}{x_1^2} + \tfrac{x_3^2z_1^2(c_1-1)}{x_2^2}\right),
\end{align*}
and by Proposition \ref{HQ2-s2}, 
\begin{align*}
(H_Q)_{g}^2(e^3_\gamma,e^3_\gamma) 
= & \tfrac{1}{x_1^2B_3} \left(\tfrac{y_1}{z_1} + \tfrac{C_3}{B_{4}} \right)^2 \tfrac{1-c_1\lambda_l}{c_1} 
 +\tfrac{1}{x_2^2B_3} \left(\tfrac{y_2A_3}{z_2} + \tfrac{C_3}{B_{4}}\right)^2 \tfrac{1-c_2\lambda_l}{c_2} \\ 
& +\tfrac{\lambda_l}{x_3^2B_3^3}\left(\tfrac{y_1}{c_1}+A_3^3\tfrac{y_2}{c_2} + \tfrac{3C_3}{B_{4}\sqrt{B_4}} \left(\tfrac{z_1}{c_1}+A_3^2\tfrac{z_2}{c_2}\right)\right)^2 \\ 
= & \tfrac{1}{x_1^2B_3} \left(\tfrac{1}{z_1} + 1\right)^2 \tfrac{1-c_1\lambda_l}{c_1} 
 +\tfrac{1}{x_2^2B_3} \left(z_1 +1\right)^2 \tfrac{(1-\tfrac{c_1}{c_1-1}\lambda_l)(c_1-1)}{c_1} \\ 
& +\tfrac{\lambda_l}{x_3^2B_3^3}\left(\tfrac{1}{c_1}-z_1^3\tfrac{-1}{c_1} + \tfrac{3}{\sqrt{B_4}} \left(\tfrac{z_1}{c_1}+z_1^2\tfrac{1}{c_1}\right)\right)^2 \\ 
= & \tfrac{1}{x_1^2B_3c_1} \tfrac{(z_1+1)^2}{z_1^2} (1-c_1\lambda_l) 
 +\tfrac{1}{x_2^2B_3c_1} \left(z_1 +1\right)^2 (1-\tfrac{c_1}{c_1-1}\lambda_l)(c_1-1) \\ 
& +\tfrac{\lambda_l}{x_3^2B_3^3c_1^2}\left(1+z_1^3 + \tfrac{3}{\sqrt{B_4}} \left(z_1+z_1^2\right)\right)^2 \\ 
= & \tfrac{1}{x_1^2B_3c_1} \tfrac{(z_1+1)^2}{z_1^2} (1-c_1\lambda_l) 
 +\tfrac{1}{x_2^2B_3c_1} \left(z_1 +1\right)^2 (c_1-1-c_1\lambda_l) \\ 
& +\tfrac{\lambda_l}{x_3^2B_3^3c_1^2}(z_1+1)^2\left(z_1^2-z_1+1 + \tfrac{3\sqrt{c_1}}{\sqrt{z_1+1}} z_1\right)^2.
\end{align*}
Thus \eqref{BRF3ra} becomes 
\begin{align*}
& \lambda_l\left((z_1+1)^2-\tfrac{x_3^2}{x_1^2} -\tfrac{x_3^2z_1^2}{x_2^2}\right) 
+\tfrac{1}{c_1}\left(\tfrac{x_3^2}{x_1^2} + \tfrac{x_3^2z_1^2(c_1-1)}{x_2^2}\right) \\
=& \tfrac{1}{x_1^2c_1} \tfrac{(z_1+1)^2}{z_1^2} (1-c_1\lambda_l) 
 +\tfrac{1}{x_2^2c_1} \left(z_1 +1\right)^2 (c_1-1-c_1\lambda_l) \\ 
& +\tfrac{\lambda_l}{x_3^2B_3^2c_1^2}(z_1+1)^2\left(z_1^2-z_1+1 + \tfrac{3\sqrt{c_1}}{\sqrt{z_1+1}} z_1\right)^2,
\end{align*}
if and only if, 
\begin{align*}
& \lambda_l\left((z_1+1)^2-\tfrac{x_3^2}{x_1^2} -\tfrac{x_3^2z_1^2}{x_2^2}\right) 
+\tfrac{1}{c_1}\left(x_3^2-\tfrac{(z_1+1)^2}{z_1^2}\right)\left(\tfrac{1}{x_1^2} + \tfrac{z_1^2(c_1-1)}{x_2^2}\right) \\
= & -\lambda_l(z_1+1)^2\left(\tfrac{1}{x_1^2z_1^2}+\tfrac{1}{x_2^2}\right) 
+\tfrac{\lambda_l}{x_3^2z_1^2}\left(z_1^2-z_1+1 + \tfrac{3\sqrt{c_1}}{\sqrt{z_1+1}} z_1\right)^2, 
\end{align*}
if and only if, 
\begin{align*}
& \lambda_l\left((z_1+1)^2+\left(\tfrac{(z_1+1)^2}{z_1^2}-x_3^2\right)\left(\tfrac{1}{x_1^2} + \tfrac{z_1^2}{x_2^2}\right)\right) 
+\tfrac{1}{c_1}\left(x_3^2-\tfrac{(z_1+1)^2}{z_1^2}\right)\left(\tfrac{1}{x_1^2} + \tfrac{z_1^2(c_1-1)}{x_2^2}\right) \\
= & \tfrac{\lambda_l}{x_3^2z_1^2}\left(z_1^2-z_1+1 + \tfrac{3\sqrt{c_1}}{\sqrt{z_1+1}} z_1\right)^2,
\end{align*}
concluding the proof.  
\end{proof}

Part (i) now follows by just checking that if $z_1=c_1-1$, then $x_1=\frac{1}{z_1}$, $x_2=1$ and $x_3=\frac{z_1+1}{z_1}$ solve equations \eqref{BRF1ra},  \eqref{BRF2da} and  \eqref{BRF3ra-2}.  It is also easy to see that if $\kg$ is abelian (i.e., only $\lambda_0=0$ appears), then the above solve the three equations for any $z_1>0$.   

If $\zg(\kg)=\kg_0\ne 0$, i.e., $\pg_3^0\ne 0$, then $x_3=\frac{z_1+1}{z_1}$ by \eqref{BRF3ra-2}, which implies that $x_1=\frac{1}{z_1}$ and $x_2=1$ by using \eqref{BRF1ra} and \eqref{BRF2da}, respectively, and so part (ii) follows.  

Under the hypothesis of part (iii), in addition to $x_1=\frac{1}{z_1}$ and $x_2=1$, we obtain from \eqref{BRF1ra} that 
$$
1-\tfrac{x_{3}}{x_1(z_1+1)} =\tfrac{x_1(z_1+1)}{x_{3}}-1,
$$
so $x_3=x_1(z_1+1)=\frac{z_1+1}{z_1}$, concluding the proof of part (iii) and the theorem.   
\end{proof}

The following is the simplest example provided by Theorem \ref{BRFggs2} and was recently given in \cite{PdsRff1}. 

\begin{example}\label{PR1-exa}
Consider $G_1=G_2=\SU(2)$ and the $1$-dimensional subgroup $K=S^1_{p,q}\subset G$ ($p\geq q$)  with the usual embedding as in \cite{PdsRff1}, which gives $c_1=\tfrac{p^2+q^2}{p^2}$.  The metric $g$ found in \cite[between (3.2) and (3.3)]{PdsRff1} for $p>q$ such that $(g,H)$ is BRF for some multiple $H$ of $H_0$ is given by 
$$
g=\left(\tfrac{q^2}{p^2+q^2},\tfrac{p^2}{p^2+q^2},1\right)_{\gk},  
$$ 
which is a multiple of the metric $g_0$ provided by Theorem \ref{BRFggs2} when $z_1=z_2=\tfrac{1}{c_1-1}$; indeed, $g_b=\tfrac{1}{c_1-1}\gk$ and   
$$
g_0=\left(c_1-1,1,c_1\right)_{g_b} = \tfrac{1}{c_1-1}\left(c_1-1,1,c_1\right)_{\gk} 
= \tfrac{p^2}{q^2}\left(\tfrac{q^2}{p^2},1,\tfrac{p^2+q^2}{p^2}\right)_{\gk}
= \tfrac{p^2+q^2}{q^2}g.    
$$ 
In the case when $p=q=1$, we have that $c_1=2$, $g_b=\gk$ and $g=\left(1,1,2\right)_{\gk}=g_0$ (see \cite[pp.11]{PdsRff1}).  
\end{example}  

We now prove the existence of a one-parameter family of BRF generalized metrics on each of the homogeneous spaces not covered by Theorem \ref{BRFggs2}, (ii) and (iii).  

\begin{theorem}\label{BRFggs2-kss}
Let $M=G/K$ be an aligned homogeneous space with $s=2$ such that Assumption \ref{assum} holds, $\kg$ is semisimple and $\cas_{\chi_{1}}=\kappa_1I_{\pg_1}$, $\cas_{\chi_{2}}=\kappa_2I_{\pg_2}$ for some $\kappa_1,\kappa_2\in\RR$ (see \eqref{Cchi}).   Consider a generalized metric $(g,H(z_1))$, where $H(z_1):=H_{Q_0}$, 
$$
g=(x_1,x_2,x_3)_{g_b}, \qquad g_b=z_1(-\kil_{\ggo_1})+ \tfrac{1}{c_1-1}(-\kil_{\ggo_2}), \qquad Q_0=\kil_{\ggo_1}-\tfrac{1}{c_1-1}\kil_{\ggo_2},
$$ 
and assume that $z_1\ne c_1-1$.
\begin{enumerate}[{\rm (i)}]
\item $(g,H(z_1))$ is Bismut Ricci flat if and only if 
\begin{align}
x_1=& \tfrac{1}{2 \kappa_1+1} \left( \kappa_1\left(\tfrac{x_3}{z_1+1} +\tfrac{z_1+1}{x_3 z_1^2}\right) + \sqrt{ \kappa_1^2\left(\tfrac{x_3}{z_1+1} +\tfrac{z_1+1}{x_3 z_1^2} \right)^2 + \tfrac{1-4\kappa_1^2}{z_1^2} } \right), \label{x1} \\ 
x_2=& \tfrac{z_1}{2 \kappa_2+1} \left( \kappa_2\left(\tfrac{x_3}{z_1+1} +\tfrac{z_1+1}{x_3 z_1^2}\right) + \sqrt{ \kappa_2^2\left(\tfrac{x_3}{z_1+1} +\tfrac{z_1+1}{x_3 z_1^2} \right)^2 + \tfrac{1-4\kappa_2^2}{z_1^2} } \right), \label{x2} 
\end{align}
$\lambda_1=\dots=\lambda_t=\lambda$, and 
\begin{equation}\label{BRF3ra-3}
px_3^4 
+ \left(\lambda(z_1+1)^2
-\tfrac{(z_1+1)^2}{z_1^2}p\right)x_3^2 
-\tfrac{\lambda}{z_1^2}\left(z_1^2-z_1+1 + \tfrac{3\sqrt{c_1}}{\sqrt{z_1+1}} z_1\right)^2 
= 0,
\end{equation}
where
$$
p:=\left(\tfrac{1}{c_1}-\lambda\right)\tfrac{1}{x_1^2} + \left(\tfrac{1}{c_2}-\lambda\right)\tfrac{z_1^2}{x_2^2}.
$$
\item If $\lambda_1=\dots=\lambda_t$ (e.g., $\kg$ simple), then for each $z_1>0$ there exists a generalized metric $(g(z_1),H(z_1))$ as in part (i) which is Bismut Ricci flat.  
\end{enumerate}
\end{theorem}

\begin{remark}\label{lowdim2}
The lowest dimensional (irreducible) example where the theorem can be applied is $M^{13}=\SU(3)\times\SU(3)/\SO(3)$.  
\end{remark}

\begin{remark}
It follows from \eqref{cij} that $p\geq 0$, where equality holds if and only if $M$ is the symmetric space $M=G\times G/\Delta G$.  
\end{remark}

\begin{remark}\label{gz0}
For $z_1=c_1-1$, it is easy to check that $x_3=\frac{c_1}{c_1-1}$ is the unique positive solution to \eqref{BRF3ra-3}, giving rise to the metric $g_0$ provided by Theorem \ref{BRFggs2}.  Indeed, it is straightforward to check that the equation is equivalent to 
$$
((c_1-1)^2 x_3^2 - c_1^2)  \left(\tfrac{(c_1-1)(1- \lambda c_2)}{x_2^2} + \tfrac{ 1- \lambda c_1}{x_1^2 (c_1-1)^2)} + \tfrac{ \lambda c_1^3}{(c_1-1)^2 x_3^2}\right),
$$
and the factor on the right is always positive by \eqref{cij}.  
\end{remark}

\begin{remark}\label{F}
If we view the expression in \eqref{BRF3ra-3} as a two variable function $F(x_3,z_1)$, then $F$ is clearly differentiable at any point $(x_3,z_1)$ with $x_3,z_1>0$.  The above remark therefore implies that as soon as $\tfrac{\partial}{\partial x_3} F\left(\frac{c_1}{c_1-1},c_1-1\right)\ne 0$, we obtain by the Implicit Function Theorem that there exists $\epsilon>0$ such that for any $z_1\in(c_1-1-\epsilon,c_1-1+\epsilon)$, the solution $x_3$ to \eqref{BRF3ra-3} is unique and it is given by a function $x_3(z_1)$ which is differentiable on $z_1$.  This gives the differentiability of the curve $(g(z_1),H(z_1))$ around $(g_0,H_0)$.  Moreover,   
\begin{equation}\label{x3p}
x_3'(c_1-1) = -\tfrac{\tfrac{\partial}{\partial z_1} F\left(\frac{c_1}{c_1-1},c_1-1\right)}{\tfrac{\partial}{\partial x_3} F\left(\frac{c_1}{c_1-1},c_1-1\right)}.  
\end{equation}
\end{remark}

\begin{remark}\label{rem2}
According to \eqref{MgBi}, $\cas_{\chi_i}=\kappa_i I_{\pg_i}$ for some $\kappa_i\in\RR$ if and only if the standard metric $\gk^i$ on the homogeneous space $M_i=G_i/\pi_i(K)$ is Einstein with $\Ricci(\gk^i)=\rho_iI_{\pg_i}$, where $\kappa_i=2\rho_i-\unm$.  In that case, $0<\kappa_i\leq \unm$, and $\kappa_i=\unm$ if and only if $G_i/\pi_i(K)$ is an irreducible symmetric space (i.e., $[\pg_i,\pg_i]\subset\pi_i(\kg)$).  Homogeneous spaces $G/K$ with $G$ simple such that the standard metric is Einstein were classified by Wang and Ziller in \cite{WngZll2}: beyond isotropy irreducible spaces, with $G$ classical, there are $10$ infinite families parametrized by the natural numbers, $2$ conceptual constructions and $2$ isolated examples, and with $G$ exceptional, $20$ isolated examples.  The isotropy $K$ is semisimple for most of them, see the tables in \cite{stab} for further information.  
\end{remark}

\begin{remark}\label{rem7}
Assume that $G_1/\pi_1(K)$ and $G_2/\pi_2(K)$ are both isotropy irreducible spaces, their isotropy representations are inequivalent and $\kg$ is simple, so $M=G/K$ is multiplicity free and for any fixed $g_b$, any $G$-invariant metric on $M=G/K$ is of the form $g=(x_1,x_2,x_3)_{g_b}$.  If we request the background metric to be fixed, say the multiple of the standard metric $g_b=\tfrac{1}{c_1-1}\gk$ (i.e., $z_1=\tfrac{1}{c_1-1}$), then each member $g(z_1)$ of the above family of metrics will be isometric to a metric of the form 
$$
g(\left(\overline{x}_1(z_1),\overline{x}_2(z_1),\overline{x}_3(z_1)\right)_{\tfrac{1}{c_1-1}\gk}, \qquad \overline{x}_i(z_1)>0,
$$
and each closed $3$-form $H(z_1)$ will pullback via the above isometry to a $3$-form of the form 
$$
\overline{H}(z_1)=H\left(\tfrac{1}{c_1-1}\right)+d\alpha(z_1), 
$$
for some $G$-invariant $2$-form $\alpha(z_1)$.  The price to pay in order to fix the reductive decomposition is therefore that the $3$-form $\overline{H}(z_1)$ is not one of the canonical ones defined by \eqref{HQ-def}.  This may turn the computations even heavier.  
\end{remark}

\begin{proof}[Proof of Theorem \ref{BRFggs2-kss}]
We first prove part (i).  Equations \eqref{BRF1ra} and \eqref{BRF2da} are equivalent to quadratic equations for $x_1$ and $x_2$, respectively, and it is easy to see that \eqref{x1} and \eqref{x2} are respectively the only positive solutions.  On the other hand, the factor multiplying $\lambda_l$ in \eqref{BRF3ra-2} is nonzero since otherwise $x_3=\tfrac{z_1+1}{z_1}$ and so $z_1=c_1-1$ by \eqref{x1}, \eqref{x2} and \eqref{BRF3ra-2}.  Thus $\lambda_1=\dots=\lambda_t=\lambda$ for some $\lambda$, and it is straightforward to show that equation \eqref{BRF3ra-2} is equivalent to 
\eqref{BRF3ra-3}, concluding the proof of (i).  

In order to prove part (ii), we fix $z_1>0$ and view the expression in \eqref{BRF3ra-3} as a function of $x_3$, say $f(x_3)$.  Using that $x_1$ and $x_2$ converges to $\infty$ and so $p\to 0$, as $x_3\to 0$, one obtains that  $f(x_3)$ converges to a negative number as $x_3\to 0$.  Furthermore, since $\tfrac{x_3}{x_1}$ and $\tfrac{x_3}{x_2}$ both converge to a number as $x_3\to\infty$, it follows that $f(x_3)$ converges to $\infty$ as $x_3\to\infty$.  This implies that there exists a positive solution $x_3$ to \eqref{BRF3ra-3} for any $z_1>0$, concluding the proof.  
\end{proof}

\begin{example}\label{PR2-exa}
The examples of BRF structures given in \cite{PdsRff2} are contained in a class which can be described as follows in terms of the notation in Theorem \ref{BRFggs2} and Theorem \ref{BRFggs2-kss}.  Given an irreducible symmetric space $G/K$, one considers the homogeneous space $M=G\times G/\Delta K$, so $G_1=G_2=G$, $c_1=2$, $\cas_{\chi_{1}}=\kappa_1I_{\pg_1}$ and $\cas_{\chi_{2}}=\kappa_2I_{\pg_2}$, where $\kappa_1=\kappa_2=\unm$.   We therefore obtain that  
$$
x_1= \tfrac{1}{2} \left(\tfrac{x_3}{z_1+1} +\tfrac{z_1+1}{x_3 z_1^2}\right)
= \tfrac{x_3^2z_1^2+(z_1+1)^2}{2(z_1+1)x_3 z_1^2},
\qquad x_2=z_1x_1, \qquad p=\tfrac{1-2\lambda}{x_1^2},
$$
and so condition \eqref{BRF3ra-3} is equivalent to the following cubic equation for $x_3^2$:
\begin{align}
& (1-2\lambda)4(z_1+1)^2 z_1^4 x_3^6 \notag \\
&+ \lambda(z_1+1)^2(x_3^2z_1^2+(z_1+1)^2)^2x_3^2  
-(z_1+1)^4(1-2\lambda)4 z_1^2x_3^4  \label{BRF-3ra-sym}\\
&-\tfrac{\lambda}{z_1^2}\left(z_1^2-z_1+1 + \tfrac{3\sqrt{2}}{\sqrt{z_1+1}} z_1\right)^2(x_3^2z_1^2+(z_1+1)^2)^2 
= 0, \notag
\end{align}
where $\lambda=\tfrac{2\dim{\kg}-\dim{\pg_1}}{4\dim{\kg}}<\unm$ by \cite[Theorem 11, pp.35]{DtrZll}.  We observe that even in this relatively simple case it is hopeless to obtain a manageable formula for the solution $x_3$ in terms of $z_1$.  Note that for $z_1=1$, the background metric $g_b$ of the BRF metric $g_0$ given by Theorem \ref{BRFggs2} is the standard metric $\gk$ and $g_0=(1,1,2)_{\gk}$.  We do not know if the BRF generalized metric found in \cite{PdsRff2} is isometric to some BRF $(g(z_1),H(z_1))$ as in the above proposition up to scaling, the methods used in \cite{PdsRff2} are very different from the approach considered in this paper.  For $5$ of the $16$ symmetric pairs $(G,K)$ with $\rank G=\rank K$ given in \cite[Table 1]{PdsRff2}, $\kg$ has a nontrivial center, so the uniqueness in Theorem \ref{BRFggs2}, (ii) holds.  It is shown in Theorem \ref{BRFcurve} below that $(g(z_1),H(z_1))$ is pairwise non-homothetic around $(g_0,H_0)$.  
\end{example}

\subsection{Families of pairwise non-homothetic BRF metrics}\label{curve-sec} 
The curve of BRF generalized metrics provided by Theorem \ref{BRFggs2}, (i) when $\kg$ is abelian has constant Ricci eigenvalues $\unc$, $\frac{c_1-1}{4}$ and $\frac{c_1}{4}$ (see Remark \ref{rem1}), so the family would most likely be pairwise isometric.  We give in the following proposition the Ricci curvature of the metrics provided by Theorem \ref{BRFggs2-kss}, with the aim of finding a {\it pairwise non-homothetic} curve $(g(t),H(t))$, $t\in (a,b)$ of BRF generalized metrics on a given homogeneous space $M=G/K$, i.e., $(g(t_1),H(t_1))$ is isometric to $(g(t_2),H(t_2))$ (up to scaling) for $t_1,t_2\in(a,b)$ if and only if $t_1=t_2$.

\begin{proposition}\label{BRFggs2-kss-ric}
If $g(z_1)$ is the metric given in Theorem \ref{BRFggs2-kss}, then for $i=1,2,3$, $\Ricci(g(z_1))|_{\pg_i}=r_i(z_1)I_{\pg_i}$, where 
$$
r_1(z_1) 
=  \tfrac{1}{4x_1z_1}\left(2\kappa_1+1 - \tfrac{2x_{3}}{x_1(z_1+1)}\kappa_1\right),  \qquad
r_2(z_1) 
=  \tfrac{c_1-1}{4x_2}\left(2\kappa_2+1 - \tfrac{2x_{3}z_1}{x_2(z_1+1)}\kappa_2\right), 
$$
and
$$
r_3(z_1) 
= \tfrac{c_1}{4x_3z_1(z_1+1)}\left(\lambda(z_1+1)^2+\left(\tfrac{1}{c_1}-\lambda\right)\tfrac{x_3^2}{x_1^2} + \left(\tfrac{1}{c_2}-\lambda\right)\tfrac{x_3^2z_1^2}{x_2^2}\right). 
$$
\end{proposition}

\begin{proof}
It follows from Proposition \ref{ricggals2} and the proof of Lemma \ref{lemma}.  
\end{proof}

Recall from Remark \ref{rem1} that for the metric $g_0=g(c_1-1)$ provided by Theorem \ref{BRFggs2} (see Remark \ref{gz0}), we have that
$$
r_1(c_1-1)=\unc, \qquad r_2(c_1-1)=\tfrac{c_1-1}{4}, \qquad r_3(c_1-1) = \tfrac{c_1}{4}\left(1-\lambda\right),  
$$
which coincide with the formulas given in Proposition \ref{BRFggs2-kss-ric} when $z_1=c_1-1$.  

Assume that the dimensions of $\pg_1,\pg_2,\pg_3$ are pairwise different.  If the metrics $g(z_1)$ and $g(z_1')$ are {\it homothetic}, i.e., isometric up to scaling, then $r_{12}(z_1)=r_{12}(z_1')$ and $r_{13}(z_1)=r_{13}(z_1')$, where 
$$
r_{12}(z_1):=\tfrac{r_1(z_1)}{r_2(z_1)}, \qquad  r_{13}(z_1):=\tfrac{r_1(z_1)}{r_3(z_1)}.  
$$
We first note that if $\kappa_1=\kappa_2$, then $x_2=z_1x_1$ by \eqref{x1} and \eqref{x2} and hence $r_{12}(z_1)=\tfrac{1}{c_1-1}$ for all $z_1>0$.  Thus the function $r_{13}$ must be used to find pairwise non-homothetic curves, and in this case $\dim{\pg_1}=\dim{\pg_2}\ne \dim{\kg}$ also works.  According to Remark \ref{F}, this will happen with $c_1-1\in (a,b)$ as soon as $r_{13}'(c_1-1)\ne 0$, that is, $g_0$ belongs to a pairwise non-homothetic curve of BRF metrics.  Unfortunately, the lack of a formula for $x_3(z_1)$ makes very difficult to compute or even estimate $r_{13}'(c_1-1)$ in a unified way.      

We now show that the curve is pairwise non-homothetic around $(g_0,H_0)$ for a huge class of homogeneous spaces.  

\begin{theorem}\label{BRFcurve}
Let $G/K$ be a compact homogeneous space such that $G$ and $K$ are both simple and $\cas_{\chi}=\kappa I_{\pg}$ for some $\kappa\in\RR$ (i.e., $\gk$ is Einstein on $G/K$).  Then there exists $\epsilon>0$ such that the curve of Bismut Ricci flat generalized metrics $(g(z_1),H(z_1))$, $z_1\in(1-\epsilon,1+\epsilon)$ on $M=G\times G/\Delta K$, as in Theorem \ref{BRFggs2-kss}, is pairwise non-homothetic .  
\end{theorem}

\begin{remark}\label{posib}
The possibilities for the space $G/K$ in the theorem are the following irreducible symmetric spaces (see \cite[7.102]{Bss}),
$$
\begin{array}{c}
\SU(n)/\SO(n),\; n\geq 3, \quad \SU(2n)/\Spe(n),\; n\geq 2, \quad \SO(n)/\SO(n-1),\; n\geq 6, \\ 
F_4/\SO(9), \quad E_6/\Spe(4), \quad E_6/F_4, \quad E_7/\SU(8), \quad E_8/\SO(16), 
 \end{array}
$$ 
the non-symmetric isotropy irreducible spaces consisting of $7$ infinite families (see \cite[7.106]{Bss}) and $24$ individual spaces (see \cite[7.107]{Bss}), and the following non-isotropy irreducible spaces (see \cite[7.108, 7.109]{Bss} or \cite[Tables 1,2,3]{stab}),
$$
\begin{array}{c}
\SO(n)/K,\; \dim{K}=n>3, \quad \SO(8)/G_2, \quad F_4/\Spin(8), \\ 
 E_7/\SO(8), \quad E_8/\SO(5), \quad E_8/\SO(9), \quad E_8/\Spin(9). 
 \end{array}
$$ 
\end{remark}

\begin{remark}\label{lowdim2}
The lowest dimensional (irreducible) examples where we obtain a pairwise non-homothetic curve are 
$$
\begin{array}{c}
M^{13}=\SU(3)\times\SU(3)/\SO(3), \qquad
M^{17}=\Spe(2)\times\Spe(2)/\SU(2), \\ 
M^{20}=\SU(4)\times\SU(4)/\Spe(2), \qquad 
M^{20}=G_2\times G_2/\SU(3).  
\end{array}
$$  
\end{remark}

\begin{remark}
The aligned homogeneous space $M=G\times G/\Delta K$ in the theorem has $c_1=2$, $\kappa_1=\kappa_2=\kappa\leq\unm$, $t=1$, $c_{11}=c_{21}=2\lambda<1$ (see \eqref{cij}), where $\lambda:=\lambda_1$.  
\end{remark}

\begin{proof}
It follows from Theorem \ref{BRFggs2-kss}, (i) that 
$$
x_1= \tfrac{1}{2 \kappa+1} \left( \kappa\left(\tfrac{x_3}{z_1+1} +\tfrac{z_1+1}{x_3 z_1^2}\right) + \sqrt{ \kappa^2\left(\tfrac{x_3}{z_1+1} +\tfrac{z_1+1}{x_3 z_1^2} \right)^2 + \tfrac{1-4\kappa^2}{z_1^2} } \right),  \qquad x_2=z_1x_1, 
$$
and from Proposition \ref{BRFggs2-kss-ric} that  
$$
r_1(z_1) 
=  \tfrac{(2\kappa+1)x_1(z_1+1) - 2\kappa x_{3}}{4x_1^2z_1(z_1+1)}, \qquad 
r_3(z_1) 
= \tfrac{\lambda(z_1+1)^2x_1^2+(1-2\lambda)x_3^2}{2x_3z_1(z_1+1)x_1^2}, 
$$
so  
$$
r_{13}(z_1) = \tfrac{x_3((2\kappa+1)x_1(z_1+1) - 2\kappa x_{3})}{2(\lambda(z_1+1)^2x_1^2+(1-2\lambda)x_3^2)}. 
$$
Note that $\dim{\pg_1}=\dim{\pg_2}=\dim{\ggo}-\dim{\kg}\ne\dim{\kg}=\dim{\pg_3}$ (see Remark \ref{posib}), so the value of the function $r_{13}$ must coincide on any two homothetic metrics.  

With the aid of a computer algebra system, it is straightforward to check that
$$
\tfrac{\partial}{\partial x_3} F\left(2,1\right)=16(1-\lambda), \qquad  
\tfrac{\partial}{\partial z_1} F\left(2,1\right)=16-10\lambda, 
$$
where $F(x_3,z_1)$ is as in Remark \ref{F}, so $x_3'(1)=-\tfrac{8-5\lambda}{8(1-\lambda)}$ by \eqref{x3p} and one obtains that 
$$
r_{13}'(1)=\tfrac{3\lambda(2\kappa(\lambda-1)+3\lambda-1)}{32(\lambda-1)^3}. 
$$  
Thus the family $(g(z_1),H(z_1))$ is pairwise non-homothetic in a neighborhood of $(g_0,H_0)$ if 
$$
2\kappa(\lambda-1)+3\lambda-1\ne 0.
$$ 
This quantity vanishes if and only if $\lambda=\frac{2\kappa+1}{2\kappa+3}$, which would imply that $2\lambda>\frac{2}{3}$ and so $\dim{\ggo}<3\dim{\kg}$ by \cite[Theorem 11, pp.35]{DtrZll}.  
According to \cite[(10)]{stab}, $\kappa=(1-2\lambda)a$, where $a:=\frac{\dim{\kg}}{\dim{\ggo}-\dim{\kg}}$, thus $a=\frac{3\lambda-1}{2(1-\lambda)(1-2\lambda)}$.  It is now easy to check that none of the spaces listed in Remark \ref{posib} satisfy these conditions.   
\end{proof}

For the space $M^{13}=\SU(3)\times\SU(3)/\SO(3)$, we have that $\kappa=\unm$ and $\lambda=\frac{1}{12}$, so according to the above proof, $r_{13}'(1)=\frac{45}{2.11^3}\approx 0.0169$, a very small number.  

In the case when $\kappa_1\ne\kappa_2$, it is clear from the formulas for $r_1$ and $r_2$ given in Proposition \ref{BRFggs2-kss-ric} that the function $r_{12}$ will most likely be non-constant, giving rise to a curve $(g(z_1),H(z_1))$, $z_1\in (a,b)$ of BRF generalized metrics which is  pairwise non-homothetic on most if not all of these homogeneous spaces. 

\begin{example}\label{curve}
We consider the homogeneous space 
$$
M^{35}=G/K=\SO(8)\times\SO(7)/G_2, 
$$ 
where the embeddings are the usual ones.  Note that the dimensions of $\pg_1$, $\pg_2$ and $\pg_3$ are respectively $14$, $7$ and $14$.  It is well know that the standard metric on the $14$-dimensional non-isotropy irreducible homogeneous space $M_1=\SO(8)/G_2$ is Einstein with Einstein constant $\rho_1=\frac{5}{12}$ and Casimir constant $\kappa_1=2\rho_1-\unm=\frac{1}{3}$ (see \cite{WngZll2} or \cite[Table 9]{stab}).  On the other hand, $M_2=\SO(7)/G_2$ is an isotropy irreducible space with $\rho_2=\frac{9}{20}$ and $\kappa_2=2\rho_2-\unm=\frac{2}{5}$ (see \cite[Table 3]{Sch}).  It follows from \cite[(10)]{stab} that $c_1\lambda=\frac{2}{3}$ and $c_2\lambda=\frac{4}{5}$ (see \eqref{cij}), which gives $c_1=\frac{11}{6}$, $c_2=\frac{11}{5}$ and $\lambda=\frac{4}{11}$.  Thus $c_1-1=\tfrac{5}{6}$ and $x_3(\tfrac{5}{6})=\tfrac{11}{5}$.  With the aid of a computer algebra system, it is straightforward to check that
$$
\tfrac{\partial}{\partial x_3} F\left(\tfrac{11}{5},\tfrac{5}{6}\right)=\tfrac{847}{90}, \qquad  
\tfrac{\partial}{\partial z_1} F\left(\tfrac{11}{5},\tfrac{5}{6}\right)=\tfrac{1994}{125}, 
$$
so $x_3'(\tfrac{5}{6})=-\tfrac{35892}{21175}\approx -1.695$ and by using \eqref{x3p} and Proposition \ref{BRFggs2-kss-ric}, one obtains that $r_{12}'(\tfrac{5}{6})=-\tfrac{864}{46585}\approx -0.0185$ and $r_{13}'(\tfrac{5}{6})=-\tfrac{2160}{41503}\approx 0.052$.  This implies that the curve $(g(z_1),H(z_1))$ of BRF generalized metrics is indeed pairwise non-homothetic around $(g_0,H_0)$.  
\end{example}

\section{BRF metrics on compact Lie groups}\label{simple-sec}

We consider in this section a connected compact Lie group $M^n=G$.  It is well known that every class in $H^3(G)$ has a unique bi-invariant representative, which in the semisimple case it is necessarily a {\it Cartan $3$-form}: 
$$
\overline{Q}(X,Y,Z):=Q([X,Y],Z), \qquad \forall X,Y,Z\in\ggo, 
$$ 
where $Q$ is a bi-invariant symmetric bilinear form on $\ggo$ (see \cite[Section 3]{H3} for further information).  Thus $b_3(G)$ is equal to the number of simple factors if $G$ is semisimple, and it is well known that Cartan $3$-forms are all harmonic with respect to any bi-invariant metric on $G$.  In contrast, in general, for a given left-invariant metric $g$, the $g$-harmonicity of a Cartan $3$-form depends on very tricky conditions in terms of the structural constants.  

We fix from now on a bi-invariant metric $g_b$ on $G$.  For any left-invariant metric $g$ on $G$, there exists a $g_b$-orthonormal basis $\{ e_1,\dots,e_n\}$ of $\ggo$ such that $g(e_i,e_j)=x_i\delta_{ij}$ for some $x_1,\dots,x_n>0$, which will be denoted by $g=(x_1,\dots,x_n)_{g_b}$.  Note that $\{ e_i/\sqrt{x_i}\}$ is a $g$-orthonormal basis of $\ggo$ with dual basis $\{ \sqrt{x_i}e_i\}$, where $e_i$ also denotes the dual basis defined by $e_i(e_j):=\delta_{ij}$.  The ordered basis $\{ e_1,\dots,e_n\}$ determines structural constants given by 
$$
[e_i,e_j]=\sum_k c_{ij}^ke_k, \qquad \mbox{or equivalently}, \quad c_{ij}^k:=g_b([e_i,e_j],e_k).      
$$
If $H_b:=g_b([\cdot,\cdot],\cdot)$, then 
\begin{equation}\label{Hk2}
(H_b)_g^2(e_k,e_l) = \sum_{i,j} \tfrac{1}{x_ix_j} g_b([e_k,e_i],e_j)g_b([e_l,e_i],e_j) 
= \sum_{i,j} \tfrac{c_{ij}^kc_{ij}^l}{x_ix_j}, \qquad\forall k,l.  
\end{equation}
Concerning Ricci curvature, it is well known that (see e.g.\ \cite[(18)]{stab-dos})
\begin{equation}\label{Rc}
\ricci(g)(e_k,e_l) = -\unm\kil_\ggo(e_k,e_l)-\unc\sum_{i,j} c_{ij}^kc_{ij}^l\tfrac{x_i^2+x_j^2-x_kx_l}{x_ix_j}, \qquad\forall k, l.
\end{equation}
Now using that 
\begin{equation}\label{cijk2}
-\kil_\ggo(e_k,e_l) = -\tr{\ad{e_k}\ad{e_l}}  = \sum_{i,j} c_{ij}^kc_{ij}^l, \qquad\forall k,l,
\end{equation}
we obtain that $\ricci(g)=\unc (H_b)_g^2$, which is a necessary condition for the generalized metric $(g,H_b)$ to be BRF, if and only 
\begin{equation}\label{BRF1}
\sum_{i,j} c_{ij}^kc_{ij}^l\tfrac{(x_i-x_j)^2-x_kx_l+1}{x_ix_j}=0, \qquad\forall k,l.
\end{equation}
Note that $g_b=(1,\dots,1)$ is always a solution to \eqref{BRF1}, in accordance to the fact that $(g_b,H_b)$ is indeed BRF.  

\begin{proposition}\label{BRF-prop}
Let $(g,H)$ be a left-invariant generalized metric on a connected compact Lie group $G$.  
\begin{enumerate}[{\rm (i)}] 
\item If the metric $g$ is bi-invariant, say $g=g_b$, and $(g_b,H)$ is BRF, then $H=\pm H_b$.  

\item If $H=\pm H_b$ for some bi-invariant metric $g_b$ and $(g,H_b)$ satisfies that $\ricci(g)=\unc (H_b)_g^2$, then $g=g_b$.  
\end{enumerate}
\end{proposition}

\begin{remark}
In other words, $(g_b,\pm H_b)$ is the only BRF generalized metric on $G$ of the form $(g,\pm H_b)$ or $(g_b,H)$, where $g$ and $H$ are respectively any left-invariant metric and any $3$-form. 
\end{remark}

\begin{proof}
We first prove part (ii).  Each equation in \eqref{BRF1} with $k=l$ can be rewritten as 
\begin{equation}\label{BRF4}
\sum_{i,j} (c_{ij}^k)^2\tfrac{(x_i-x_j)^2}{x_ix_j}=(x_k^2-1)\sum_{i,j} (c_{ij}^k)^2\tfrac{1}{x_ix_j}, 
\end{equation}
which implies that for any $k$, $x_k\geq 1$ and the number $x_k^2-1$ belongs to the convex hull of $\{ (x_i-x_j)^2:1\leq i,j\leq n\}$.  In particular, if say, $x_1\leq\dots\leq x_n$, then $x_n^2-1\leq (x_1-x_n)^2$ and so $x_1\leq x_n\leq \frac{x_1^2+1}{2x_1}\leq x_1$.  Hence $x_1=\dots=x_n=1$, i.e., $g=g_b$.  

On the other hand, to prove part (i), we consider the decomposition 
$$
\ggo=\ggo_1\oplus\dots\oplus\ggo_r\oplus\zg(\ggo),
$$ 
where $\ggo_i$ is a simple ideal for all $i$ and assume that $(g_b,H)$ is BRF.  We have that $g_b=z_1g_1+\dots+z_rg_r+g_0$ for some $z_i>0$ and metric $g_0$ on $\zg(\ggo)$, where $g_i:=-\kil_{\ggo_i}$.  Since $H$ is $g_b$-harmonic, $H$ is necessarily  bi-invariant and so $H=y_1H_1+\dots+y_rH_r+H_0$ for some $y_i\in\RR$ and $3$-form $H_0$ on $\zg(\ggo)$, where $H_i:=-\kil_{\ggo_i}([\cdot,\cdot],\cdot)$.  Using that $(z_ig_i,z_iH_i)$ is BRF on $G_i$ for all $i$, we obtain that 
\begin{align*}
\unc g_1+\dots+\unc g_r =& \ricci(g_b) = \unc(H)_{g_b}^2  
= \unc y_1^2(H_1)_{z_1g_1}^2+\dots+\unc y_r^2(H_r)_{z_rg_r}^2+\unc(H_0)_{g_0}^2\\
=& \tfrac{y_1^2}{z_1^2}\ricci(z_1g_1)+\dots+\tfrac{y_r^2}{z_r^2}\ricci(z_rg_r)+\unc(H_0)_{g_0}^2\\
=& \unc\tfrac{y_1^2}{z_1^2}g_1+\dots+\unc\tfrac{y_r^2}{z_r^2}g_r+\unc(H_0)_{g_0}^2,
\end{align*}
so $z_i=y_i$ for all $i$ and $H_0=0$, that is, $H=g_b([\cdot,\cdot],\cdot)$.  
\end{proof}

\appendix

\section{Corrigendum}\label{app} 

This appendix is the content of \cite{BRF-C} and its purpose is to report an incorrect formula in the previous arXiv version of this paper (i.e., Sections 1-6), which was published as \cite{BRF}.  The consequences of this mistake for the statement and the proof of the main theorem are also worked out.  
%Any reference to the article \cite{BRF} can be read as a reference to the present paper with the same numbering; for example, \cite[Theorem 5.1]{BRF} is Theorem \ref{BRFggs2}.     

There is a mistake in the formula for $(H_Q)_{g}^2(e^3_\gamma,e^3_\gamma)$ given in Proposition \ref{HQ2-s2}, the number $\sqrt{B_4}$ must be deleted from the last summand.  This is due to an incorrect use of \eqref{L71} in replacing $Q(e^4_\beta,e^3_\beta)$ in the ninth line of the proof.  The correct statement is the following. 

\begin{proposition}\label{prop} (cf.\ Proposition \ref{HQ2-s2}).
If $e^3_\gamma\in\pg_3^l$, where $l=0,1,\dots,t$, then 
\begin{align*}
(H_Q)_{g}^2(e^3_\gamma,e^3_\gamma) 
= & \tfrac{1}{x_1^2B_3} \left(\tfrac{y_1}{z_1} + \tfrac{C_3}{B_{4}} \right)^2 \tfrac{1-c_1\lambda_l}{c_1} 
 +\tfrac{1}{x_2^2B_3} \left(\tfrac{y_2A_3}{z_2} + \tfrac{C_3}{B_{4}}\right)^2 \tfrac{1-c_2\lambda_l}{c_2} \\ 
& +\tfrac{\lambda_l}{x_3^2B_3^3}\left(\tfrac{y_1}{c_1}+A_3^3\tfrac{y_2}{c_2} + \tfrac{3C_3}{B_{4}} \left(\tfrac{z_1}{c_1}+A_3^2\tfrac{z_2}{c_2}\right)\right)^2.  
\end{align*}
\end{proposition}

As a consequence, the generalized metrics provided by Theorem \ref{BRFggs2}, (i) are actually BRF for any $z_1>0$ and the uniqueness results given in parts (ii) and (iii) always hold, without any condition on $\kg$ or the Casimir operators $\cas_{\chi_i}$.  The theorem can now be stated and proved as follows.    

\begin{theorem}\label{thm} (cf.\ Theorem \ref{BRFggs2}).
Let $M=G/K$ be an aligned homogeneous space with $s=2$ such that Assumption \ref{assum} holds.   
\begin{enumerate}[{\rm (i)}]
\item The $G$-invariant generalized metric $(g_0(z_1),H_{0}(z_1))$ defined by $H_{0}(z_1):=H_{Q_0}$, 
\begin{align}
g_0(z_1):=\left(\tfrac{1}{z_1},1,\tfrac{z_1+1}{z_1}\right)_{g_b}, \quad g_b:=z_1(-\kil_{\ggo_1})+ \tfrac{1}{c_1-1}(-\kil_{\ggo_2}), \quad Q_0:=\kil_{\ggo_1}-\tfrac{1}{c_1-1}\kil_{\ggo_2}, \label{gthm}
\end{align}  
is Bismut Ricci flat for any $z_1>0$.

\item The structures $(g_0(z_1),H_{0}(z_1))$, $z_1>0$ are the only $G$-invariant Bismut Ricci flat generalized metrics on $M=G/K$ up to scaling of the form $(g=(x_1,x_2,x_3)_{g_b},H_Q)$.   
\end{enumerate}
\end{theorem}

\begin{remark}\label{rem000}
It is not hard to prove that each metric $g_0(z_1)$ can also be written diagonally in terms of the standard metric, i.e., $g_0(z_1)=\left(y_1,y_2,y_3\right)_{\gk}$.  Since $H_0:=H_0(\tfrac{1}{c_1-1})$ is the unique (up to scaling) harmonic $3$-form for any of these metrics, it follows from part (ii) that 
$$
g_0(z_1)=g_0(\tfrac{1}{c_1-1}) =\left(c_1-1,1,c_1\right)_{g_b} = \left(1,\tfrac{1}{c_1-1},\tfrac{c_1}{c_1-1}\right)_{\gk} =:g_0, \qquad \forall z_1>0.     
$$
Thus the generalized metrics in part (i) are actually a single one: $(g_0,H_0)$.  Note that in terms of the $\gk$-orthogonal reductive decomposition, $H_0=H_{Q_0}$.  
\end{remark}

\begin{proof}[Proof of Theorem \ref{thm}]
The proof remains exactly the same as the proof of Theorem \ref{BRFggs2} until Lemma \ref{lemma}.  On the other hand, it is easy to see that if $\cas_{\chi_{1}}=\kappa_1I_{\pg_1}$ for some $\kappa_1\in\RR$, then
\begin{align}
x_1=& \tfrac{1}{2 \kappa_1+1} \left( \kappa_1\left(\tfrac{x_3}{z_1+1} +\tfrac{z_1+1}{x_3 z_1^2}\right) + \sqrt{ \kappa_1^2\left(\tfrac{x_3}{z_1+1} +\tfrac{z_1+1}{x_3 z_1^2} \right)^2 + \tfrac{1-4\kappa_1^2}{z_1^2} } \right), \label{x1C} 
\end{align}
and if $\cas_{\chi_{2}}=\kappa_2I_{\pg_2}$, then
\begin{align}
x_2=& \tfrac{z_1}{2 \kappa_2+1} \left( \kappa_2\left(\tfrac{x_3}{z_1+1} +\tfrac{z_1+1}{x_3 z_1^2}\right) + \sqrt{ \kappa_2^2\left(\tfrac{x_3}{z_1+1} +\tfrac{z_1+1}{x_3 z_1^2} \right)^2 + \tfrac{1-4\kappa_2^2}{z_1^2} } \right). \label{x2C} 
\end{align}
The above two formulas are precisely \eqref{x1} and \eqref{x2}, respectively.  

\begin{lemma}\label{lemmaC} (cf.\ Lemma \ref{lemma}).
The third set of equations \eqref{BRF3ra} is equivalent to 
\begin{equation}\label{eq}
\left(x_3^2-\tfrac{(z_1+1)^2}{z_1^2}\right)\left(\lambda_l\tfrac{(z_1+1)^2}{x_3^ 2}
+\left(\tfrac{1}{c_1}-\lambda_l\right)\tfrac{1}{x_1^2} 
+\left(\tfrac{1}{c_2}-\lambda_l\right)\tfrac{z_1^2}{x_2^2}\right) = 0, 
\end{equation}
for any $l=0,1,\dots,t$. 
\end{lemma}

\begin{proof}[Proof of Lemma \ref{lemmaC}]
The computation of $r_l$ given in the paper is correct and by Proposition \ref{prop}, 
\begin{align*}
(H_Q)_{g}^2(e^3_\gamma,e^3_\gamma) 
= & \tfrac{1}{x_1^2B_3} \left(\tfrac{y_1}{z_1} + \tfrac{C_3}{B_{4}} \right)^2 \tfrac{1-c_1\lambda_l}{c_1} 
 +\tfrac{1}{x_2^2B_3} \left(\tfrac{y_2A_3}{z_2} + \tfrac{C_3}{B_{4}}\right)^2 \tfrac{1-c_2\lambda_l}{c_2} \\ 
& +\tfrac{\lambda_l}{x_3^2B_3^3}\left(\tfrac{y_1}{c_1}+A_3^3\tfrac{y_2}{c_2} + \tfrac{3C_3}{B_{4}} \left(\tfrac{z_1}{c_1}+A_3^2\tfrac{z_2}{c_2}\right)\right)^2 \\ 
= & \tfrac{1}{x_1^2B_3} \left(\tfrac{1}{z_1} + 1\right)^2 \tfrac{1-c_1\lambda_l}{c_1} 
 +\tfrac{1}{x_2^2B_3} \left(z_1 +1\right)^2 \tfrac{(1-\tfrac{c_1}{c_1-1}\lambda_l)(c_1-1)}{c_1} \\ 
& +\tfrac{\lambda_l}{x_3^2B_3^3}\left(\tfrac{1}{c_1}-z_1^3\tfrac{-1}{c_1} + 3 \left(\tfrac{z_1}{c_1}+z_1^2\tfrac{1}{c_1}\right)\right)^2 \\ 
= & \tfrac{1}{x_1^2B_3c_1} \tfrac{(z_1+1)^2}{z_1^2} (1-c_1\lambda_l) 
 +\tfrac{1}{x_2^2B_3c_1} \left(z_1 +1\right)^2 (1-\tfrac{c_1}{c_1-1}\lambda_l)(c_1-1) \\ 
& +\tfrac{\lambda_l}{x_3^2B_3^3c_1^2}\left(1+z_1^3 + 3 \left(z_1+z_1^2\right)\right)^2 \\ 
= & \tfrac{1}{x_1^2B_3c_1} \tfrac{(z_1+1)^2}{z_1^2} (1-c_1\lambda_l) 
 +\tfrac{1}{x_2^2B_3c_1} \left(z_1 +1\right)^2 (c_1-1-c_1\lambda_l) \\ 
& +\tfrac{\lambda_l}{x_3^2B_3^3c_1^2}(z_1+1)^2\left(z_1^2-z_1+1 + 3z_1\right)^2.
\end{align*}
Thus \eqref{BRF3ra} becomes 
\begin{align*}
& \lambda_l\left((z_1+1)^2-\tfrac{x_3^2}{x_1^2} -\tfrac{x_3^2z_1^2}{x_2^2}\right) 
+\tfrac{1}{c_1}\left(\tfrac{x_3^2}{x_1^2} + \tfrac{x_3^2z_1^2(c_1-1)}{x_2^2}\right) \\
=& \tfrac{1}{x_1^2c_1} \tfrac{(z_1+1)^2}{z_1^2} (1-c_1\lambda_l) 
 +\tfrac{1}{x_2^2c_1} \left(z_1 +1\right)^2 (c_1-1-c_1\lambda_l)  
+\tfrac{\lambda_l}{x_3^2B_3^2c_1^2}(z_1+1)^6,
\end{align*}
if and only if, 
\begin{align*}
& \lambda_l\left((z_1+1)^2-\tfrac{x_3^2}{x_1^2} -\tfrac{x_3^2z_1^2}{x_2^2}\right) 
+\tfrac{1}{c_1}\left(x_3^2-\tfrac{(z_1+1)^2}{z_1^2}\right)\left(\tfrac{1}{x_1^2} + \tfrac{z_1^2(c_1-1)}{x_2^2}\right) \\
= & -\lambda_l(z_1+1)^2\left(\tfrac{1}{x_1^2z_1^2}+\tfrac{1}{x_2^2}\right) 
+\tfrac{\lambda_l}{x_3^2z_1^2}(z_1+1)^4, 
\end{align*}
if and only if, 
\begin{align*}
& \lambda_l\left((z_1+1)^2+\left(\tfrac{(z_1+1)^2}{z_1^2}-x_3^2\right)\left(\tfrac{1}{x_1^2} + \tfrac{z_1^2}{x_2^2}\right)\right) 
+\tfrac{1}{c_1}\left(x_3^2-\tfrac{(z_1+1)^2}{z_1^2}\right)\left(\tfrac{1}{x_1^2} + \tfrac{z_1^2(c_1-1)}{x_2^2}\right) \\
=&  \tfrac{\lambda_l}{x_3^2z_1^2}(z_1+1)^4,
\end{align*}
if and only if, 
\begin{align*}
& \lambda_l\left(\tfrac{(z_1+1)^2}{x_3^ 2}\left(x_3^2-\tfrac{(z_1+1)^2}{z_1^2}\right)
+ \left(\tfrac{(z_1+1)^2}{z_1^2}-x_3^2\right)\left(\tfrac{1}{x_1^2} + \tfrac{z_1^2}{x_2^2}\right) 
\right) \\
&+\tfrac{1}{c_1}\left(x_3^2-\tfrac{(z_1+1)^2}{z_1^2}\right)\left(\tfrac{1}{x_1^2} + \tfrac{z_1^2(c_1-1)}{x_2^2}\right) = 0, 
\end{align*}
if and only if, 
$$
\left(x_3^2-\tfrac{(z_1+1)^2}{z_1^2}\right)\left(\lambda_l\left(\tfrac{(z_1+1)^2}{x_3^ 2}
-\tfrac{1}{x_1^2} - \tfrac{z_1^2}{x_2^2}\right) 
+\tfrac{1}{c_1}\left(\tfrac{1}{x_1^2} + \tfrac{z_1^2(c_1-1)}{x_2^2}\right)\right) = 0, 
$$
concluding the proof of the lemma (recall that $c_2=\tfrac{c_1}{c_1-1}$).  
\end{proof}

We now continue with the proof of Theorem \ref{thm}.  Part (i) of the theorem now follows by just checking that for any $z_1>0$, the numbers $x_1=\frac{1}{z_1}$, $x_2=1$ and $x_3=\frac{z_1+1}{z_1}$ solve equations \eqref{BRF1ra},  \eqref{BRF2da} and  \eqref{eq}.   

Since the right hand factor of \eqref{eq} is positive by \eqref{cij} (note that $c_i\lambda_l=c_{il}<1$), we obtain that $x_3=\frac{z_1+1}{z_1}$ must always hold, which by \eqref{x1C} and \eqref{x2C} implies that $x_1=\frac{1}{z_1}$ and $x_2=1$, respectively, independently of whether $\cas_{\chi_i}$ is a multiple of the identity or not.  This concludes the proof of part (ii) of the theorem.   
\end{proof}

The rest of Section 5 dealing with the existence of curves, starting from Theorem \ref{BRFggs2-kss}, becomes obsolete, with the only exception of Example \ref{PR2-exa}, which can be rewritten as follows.   

\begin{example}\label{exa} (cf.\ Example \ref{PR2-exa}).
The examples of BRF structures given in \cite{PdsRff2} are contained in a class which can be described as follows in terms of the notation in Theorem \ref{thm}.  Given an irreducible symmetric space $G/K$, one considers the homogeneous space $M=G\times G/\Delta K$, so $G_1=G_2=G$ and $c_1=2$.   We therefore obtain for $z_1=1$ that $g_0=(1,1,2)_{\gk}$.  We do not know if the BRF generalized metric found in \cite{PdsRff2} is isometric to $(g_0,H_0)$ up to scaling, the methods used in \cite{PdsRff2} are very different from the approach considered in this paper.    
\end{example}

\begin{remark}\label{lowdimC}
In addition to the examples given in Remark \ref{lowdim}, the following aligned homogeneous spaces are also covered by Theorem \ref{thm}:
\begin{enumerate}[\small{$\bullet$}]
\item $M^{13}=\SU(3)\times\SU(3)/\SO(3)$, $c_1=2$,    

\item $M^{17}=\Spe(2)\times\Spe(2)/\SU(2)$, $c_1=2$,

\item $M^{20}=\SU(4)\times\SU(4)/\Spe(2)$, $c_1=2$,
 
\item $M^{20}=G_2\times G_2/\SU(3)$, $c_1=2$,

\item $M^{21}=G_2\times \Spe(2)/\SU(2)$, $c_1=\tfrac{71}{56}$,

\item $M^{35}=\SO(8)\times \SO(7)/G_2$, $c_1=\tfrac{11}{6}$, 

\item $M^{50}=\SO(10)\times \SU(4)/\Spe(2)$, $c_1=\tfrac{7}{6}$, 

\item $M^{55}=\SU(7)\times \SO(8)/\SO(7)$, $c_1=\tfrac{10}{7}$.
\end{enumerate}
\end{remark}

\begin{remark}\label{mistake}
After submitting paper \cite{BRF}, we continued working on the Ricci curvature of aligned homogeneous spaces.  What we discovered is that all metrics $g_0(z_1)$, $z_1>0$ in Theorem \ref{BRFggs2} (see also Theorem \ref{thm}) can be written as diagonal metrics in terms of a single reductive decomposition, as explained in Remark \ref{rem000}.  This led us to a contradiction between the uniqueness results given in Theorem \ref{BRFggs2}, (ii) and (iii) and Theorem \ref{BRFcurve}, which states that such a curve is pairwise non-homothetic.  We therefore went through the paper to check line by line the computations and found, after a couple of months of redoing everything, the `silly' mistake in the proof of Proposition \ref{HQ2-s2}.  The new formula converted ugly equation \eqref{BRF3ra-2} into the more natural condition \eqref{eq} and everything smoothed out from there.  The expected fact that the curve was actually the same metric viewed in terms of different reductive decompositions became clear.  
\end{remark}

\begin{remark}\label{mistake2}
In our paper in preparation on the Ricci curvature of aligned homogeneous spaces, we have computed Ricci in a different way, using structural constants, as well as Ricci of normal metrics using classical formulas.  Fortunately, the results we obtained coincide with Proposition \ref{ricggals2}, so we are fully confident about Ricci formulas.  The computations on homogeneous spaces $G/K$ with $G$ non-simple are really long and complicated, this is why in most of the literature on compact homogeneous Einstein manifolds it is assumed that $G$ is simple, the case $G$ non-simple is practically unexplored.  
\end{remark}


\begin{thebibliography}{MM}

\bibitem[AF]{AgrFrd} {\sc I. Agricola, T. Friedrich}, A note on flat metric connections with antisymmetric torsion, {\it Diff. Geom. Appl.} {\bf 28} (2010), 480-487.

\bibitem[B]{Bss} {\sc A. Besse}, Einstein manifolds, {\it Ergeb. Math.} {\bf 10} (1987), Springer-Verlag, Berlin-Heidelberg.

\bibitem[CK]{CrtKrs} {\sc V. Cortes, D. Krusche}, Classification of generalized Einstein metrics on 3-dimensional Lie groups, preprint 2022 (arXiv).  

\bibitem[DZ]{DtrZll} {\sc J. D'Atri, W. Ziller}, Naturally reductive metrics and Einstein metrics on compact lie groups, {\it Mem. Amer. Math. Soc.} {\bf 215} (1979).

\bibitem[GF]{Grc} {\sc M. Garcia-Fernandez}, Ricci flow, Killing spinors, and T-duality in generalized geometry, {\it Adv. Math.} {\bf 350} (2019), 1059-1108.

\bibitem[GFS]{GrcStr} {\sc M. Garcia-Fernandez, J. Streets}, Generalized Ricci Flow, {\it AMS University Lecture Series} {\bf 76}, 2021.  

\bibitem[LL]{stab} {\sc E.A. Lauret, J. Lauret}, The stability of standard homogeneous Einstein manifolds,  {\it Math. Z.} {\bf 303}, 16 (2023).   

\bibitem[LW1]{stab-dos} {\sc J. Lauret, C.E. Will}, On the stability of homogeneous Einstein manifolds II, {\it J. London Math. Soc.} {\bf 106} (2022), 3638-3669.

\bibitem[LW2]{H3}  {\sc J. Lauret, C.E. Will}, Harmonic $3$-forms on compact homogeneous spaces, {\it J. Geom. Anal.}, in press (arXiv).  

\bibitem[Le]{Lee} {\sc Kuan-Hui Lee}, The Stability of Generalized Ricci Solitons, preprint 2022 (arXiv).  

\bibitem[PR1]{PdsRff1} {\sc F. Podest\`a, A. Raffero}, Bismut Ricci flat manifolds with symmetries, {\it Proc. Royal Soc. Edinburgh: Sec. A, Math.}, in press.  

\bibitem[PR2]{PdsRff2} {\sc F. Podest\`a, A. Raffero}, Infinite families of homogeneous Bismut Ricci flat manifolds, {\it Comm. Contemp. Math.}, in press.  

\bibitem[RT]{RubTpl} {\sc Roberto Rubio and Carl Tipler}, The Lie group of automorphisms of a courant algebroid and the moduli space of generalized metrics,  {\it Rev. Mat. Iberoamer.}, {\bf 36} (2019), 485-536.

\bibitem[S]{Sch} {\sc P. Schwahn}, The Lichnerowicz Laplacian on normal homogeneous spaces, preprint 2023 (arXiv).

\bibitem[S1]{Str1} {\sc J. Streets}, Regularity and expanding entropy for connection Ricci flow, {\it J. Geom. Phys.} {\bf 58} (2008), 900-912.  

\bibitem[S2]{Str2} {\sc J. Streets}, Generalized geometry, T-duality, and renormalization group flow, {\it J. Geom. Phys.} {\bf 114} (2017), 506-522. 

\bibitem[WZ]{WngZll2} {\sc M. Y. Wang, W. Ziller}, On normal homogeneous Einstein manifolds, {\it Ann.
Sci. \'Ecole Norm. Sup.} {\bf 18} (1985), 563-633.
\end{thebibliography}

\begin{thebibliography}{MM}

\bibitem[LW3]{BRF}  {\sc J. Lauret, C.E. Will}, Bismut Ricci flat generalized metrics on compact homogeneous spaces, {\it Trans. Amer. Math. Soc.} {\bf 376} (2023), 7495-7519.  

\bibitem[LW4]{BRF-C}  {\sc J. Lauret, C.E. Will}, Corrigendum to "Bismut Ricci flat generalized metrics on compact homogeneous spaces", {\it Trans. Amer. Math. Soc.} {\bf 376} (2023), 7495-7519, preprint 2023.   
\end{thebibliography}
\end{document}